%% file: paper1.tex
\documentclass{elsarticle}

\usepackage[utf8]{inputenc}
\usepackage{amsmath}
\usepackage{amssymb}
\usepackage{amsfonts}
\usepackage{amsthm}
\usepackage{enumitem}
\usepackage{geometry}
\usepackage{hyperref}
\usepackage{natbib}
\usepackage{subcaption}
\usepackage{booktabs}
\usepackage{longtable}
\usepackage{siunitx}
\usepackage{color}
\usepackage[normalem]{ulem}
\usepackage{graphicx}
\usepackage{tikz}
\usetikzlibrary{patterns}
\usepackage{multirow}

\newif\iftracked
\ifdefined\TRACKED\trackedtrue\else\trackedfalse\fi
\newcommand{\revision}[2]{\iftracked\par\noindent{\color{red}\begingroup\def\cite##1{[cite]}\def\eqref##1{[ref]}\sout{#1}\endgroup}\par\smallskip\noindent{\color{blue}#2}\par\smallskip\else#2\fi}

  \newtheorem{theorem}{Theorem}
  \newtheorem{lemma}[theorem]{Lemma}
  \newtheorem{proposition}[theorem]{Proposition}
  \newtheorem{corollary}[theorem]{Corollary}
  \newtheorem{assumption}[theorem]{Assumption}
  
  \newtheorem{remark}{Remark}

\newcommand{\R}{\mathbb{R}}
\newcommand{\eps}{\epsilon}

\newcommand{\OmM}{\Omega_m}
\newcommand{\norm}[1]{\|#1\|}
\newcommand{\vertiii}[1]{\left|\!\left|\!\left| #1 \right|\!\right|\!\right|}
\newcommand{\fint}{\mathop{\int}\limits}

\begin{document}

\makeatletter
\def\ps@pprintTitle{%
    \let\@oddhead\@empty
    \let\@evenhead\@empty
    \let\@oddfoot\@empty
    \let\@evenfoot\@empty}
\makeatother

\begin{frontmatter}
    \title{Uniform Inf-Sup Norm Equivalence and Robust Operator Preconditioning for Stokes Flow in tight domains with Periodic Pillars}
    \author[inst1,inst2]{Qi Xin}
    \ead{qixin1@link.cuhk.edu.cn}

    \author[inst2]{Yan Xie}
    \ead{xieyan@sribd.cn}

    \author[inst5]{Chen-song Zhang}
    \ead{zhangcs@lsec.cc.ac.cn}

    \author[inst1,inst2,inst6]{Shihua Gong\corref{cor1}}
    \ead{gongshihua@cuhk.edu.cn}

    \author[inst4]{Jinchao Xu}
    \ead{jinchao.xu@kaust.edu.sa}

    \cortext[cor1]{Corresponding author}

    \affiliation[inst1]{organization={School of Science and Engineering, The Chinese University of Hong Kong},
        addressline={Shenzhen},
        postcode={518172},
        state={Guangdong},
        country={China}}

    \affiliation[inst2]{organization={Shenzhen International Center for Industrial and Applied Mathematics, Shenzhen Research Institute of Big Data},
        addressline={Shenzhen},
        postcode={518172},
        state={Guangdong},
        country={China}}

    \affiliation[inst5]{organization={SKLMS, Academy of Mathematics and Systems Science, Chinese Academy of Sciences; School of Mathematical Sciences, University of Chinese Academy of Sciences},
        city={Beijing},
        postcode={100049},
        country={China}}
    
    \affiliation[inst4]{organization={Applied Mathematics and Computational Sciences, CEMSE Division, King Abdullah University of Science and Technology},
        city={Thuwal},
        postcode={23955},
        country={Saudi Arabia}}

    \affiliation[inst6]{organization={Shenzhen Loop Area Institute},
        addressline={Shenzhen},
        postcode={518172},
        state={Guangdong},
        country={China}}


    \begin{keyword}
        Stokes equations \sep Perforated domains \sep Operator preconditioning \sep Homogenization 
    \end{keyword}

    \begin{abstract}
        Many microfluidic and porous-media computations reduce to the same core task: solving a Stokes saddle-point system on a domain perforated by a dense periodic array of pillars, as in deterministic lateral displacement (DLD) particle sorters. After rescaling the device to unit size, the geometry is controlled by a single dimensionless parameter $m$---the number of pillars across the device, equal to the inverse period. In realistic devices $m$ reaches the hundreds or thousands, and as it grows the Stokes inf-sup constant decays like $m^{-1}$, the pressure Schur complement becomes severely ill-conditioned, and standard block solvers slow down in proportion to the pillar density.

        We remove this bottleneck by identifying the pressure norm that the divergence operator induces on such geometries. For periodic pillar arrays in the proportional-hole regime, we prove that this inf-sup norm is uniformly equivalent to the $L^2+\sigma_\eps H^1$ $K$-functional norm at the pore scale $\sigma_\eps\asymp\eps$, with constants independent of the period, the pillar count $m$, and the mesh size $h$. The equivalence identifies the perforated Stokes problem with a Brinkman problem at a homogenized permeability, and its Riesz map reduces to a pressure-mass inverse plus a scaled stiffness inverse. Combined with operator preconditioning, this yields a block preconditioner built from standard algebraic-multigrid solves whose iteration count is essentially independent of both mesh size and pillar density. Two-dimensional Taylor--Hood experiments confirm the predicted robustness in mesh refinement, pillar density, geometric scale, close packing, and time step.
    \end{abstract}
\end{frontmatter}

\section{Introduction}\label{sec:intro}

Deterministic lateral displacement (DLD) devices separate particles by guiding a viscous flow through a periodic array of micron-scale pillars \cite{huang2004continuous}. Their design requires resolving the pore-scale Stokes flow: streamlines determine particle deflection, while pressure drop and shear determine operating conditions. The same computation arises in permeability calculations and in pore-scale models of fibrous and porous media \cite{Lu2020}. In each case, the governing linear system is a Stokes saddle-point problem posed on a large, densely perforated domain.

We consider a macroscopic domain $\Omega$ containing a periodic pillar array and write $\Omega_m$ for its fluid part. The period is $\eps$, the obstacle size is $a_\eps$, and $m\asymp\eps^{-1}$ is the number of cells across a unit-size device. Throughout this introduction, $d\in\{2,3\}$ and $q\in L_0^2(\Omega_m)$. At the discrete level, $A$ denotes the velocity block, $B$ the divergence matrix, and $S=BA^{-1}B^\top$ the pressure Schur complement.

The difficulty is geometric. A pillar array is a multiply connected pressure network, not a collection of independent pores: pressure information passes through a long sequence of narrow throats. In the proportional-hole regime, the continuous LBB constant satisfies $\beta(\Omega_m)=\Theta(m^{-1})$ \cite{infsup_perforation}. Consequently, $\kappa(S)=\Theta(m^2)$ and standard block preconditioners acquire iteration counts that grow with the number of pillars. This is not a mesh-resolution effect; it is the operator-level imprint of the perforated geometry.

Operator preconditioning reduces the issue to identifying the pressure norm induced by the divergence operator \cite{Mardal2010,Bramble1988,loghin2004analysis}:
\[
    \|q\|_{*,\Omega_m}:=\sup_{\mathbf v\in H_0^1(\Omega_m)^d\setminus\{\mathbf0\}}
    \frac{(q,\nabla\cdot\mathbf v)_{\Omega_m}}{\|\nabla\mathbf v\|_{L^2(\Omega_m)}}.
\]
The Schur complement is the Galerkin realization of this norm squared. Thus a pressure norm uniformly equivalent to $\norm{\cdot}_{*,\Omega_m}$ immediately specifies a robust Schur-complement block.

Our main result gives such a characterization for periodic pillar arrays. Define
\[
 \norm{q}_{L^2+\sigma_\eps H^1}:=\inf_{\tilde q\in H^1(\Omega_m)}
 \bigl(\norm{q-\tilde q}_{L^2(\Omega_m)}^2+
       \sigma_\eps^2\norm{\nabla\tilde q}_{L^2(\Omega_m)}^2\bigr)^{1/2},
\]
where $\sigma_\eps=\eps|\log(a_\eps/\eps)|^{1/2}$. For proportional holes, $\sigma_\eps\asymp\eps$. We prove the two-sided equivalence
\begin{equation}\label{eq:intro_equivalence}
    c\,\norm{q}_{L^2+\sigma_\eps H^1}
    \;\le\;
    \norm{q}_{*,\Omega_m}
    \;\le\;
    C\,\norm{q}_{L^2+\sigma_\eps H^1},
    \qquad \forall q\in L_0^2(\Omega_m),
\end{equation}
with constants $c,C>0$ independent of $\eps$ (and depending only on the fixed outer domain and reference-cell geometry). The theorem identifies the precise scale at which mass and pressure-gradient effects must be combined. It upgrades the uniform pressure estimates used in homogenization \cite{Allaire1991a,Allaire1991,Lu2020} to a norm equivalence valid for every pressure test function.

The proof is tailored to perforated networks. Its central tool is the homogenization restriction operator $R_\eps:H_0^1(\Omega)^d\to H_0^1(\Omega_m)^d$ \cite{Allaire1991a,Allaire1991,Lu2020}. Dualizing this operator yields a pressure extension from the perforated domain to the filled domain. A weighted Ne\v{c}as estimate on the fixed domain then produces the lower bound in \eqref{eq:intro_equivalence}; the upper bound follows directly from integration by parts and Poincar\'e's inequality. This mechanism is native to a periodic pillar network and isolates the close-packing deterioration in the restriction-operator constant.

The resulting Riesz map has a simple discrete form,
\begin{equation}\label{eq:intro_riesz}
    M_K^{-1}=M_p^{-1}+\sigma_\eps^{-2}L_p^{-1},
\end{equation}
where $M_p$ and $L_p$ are the pressure mass and Neumann stiffness matrices. Equation~\eqref{eq:intro_riesz} leads to a pressure block assembled from a mass inverse and a scaled stiffness inverse. Together with operator preconditioning, it yields a block preconditioner whose quality is independent of the number of perforations, subject to the stated discrete stability assumptions.

The thin-channel result of Sande, Koch, Kuchta and Mardal \cite{sande2025robust} shares the same $L^2+\sigma H^1$ functional form, but addresses a different geometry using a channel-specific coarse partition and quasi-interpolation argument. Our restriction--extension construction instead treats a multiply connected periodic network. In a complementary direction, Meier, B\"ansch and Frank \cite{Meier2022} proposed combining pressure mass and stiffness components for channel-dominated flows. Equation~\eqref{eq:intro_riesz} supplies the geometry-dependent scale and the norm equivalence that make this combination robust for periodic pillar arrays.

The contributions are:
\begin{enumerate}
    \item \textbf{Sharp continuous pressure norm.} Theorem~\ref{thm:main} proves the uniform equivalence \eqref{eq:intro_equivalence} with explicit scale $\sigma_\eps$. The proof is a PDE-level restriction--extension argument rather than a cellwise construction.
    \item \textbf{Discrete realization and robust block preconditioning.} Under the stated Fortin assumption, Corollary~\ref{cor:discrete_equiv} transfers the norm equivalence to stable finite-element pairs. Section~\ref{subsec:riesz_derivation} derives \eqref{eq:intro_riesz} and the associated triangular preconditioner.
    \item \textbf{Numerical validation in pillar arrays.} Section~\ref{sec:numerical} tests mesh refinement, pillar density, geometric scale, close packing, and the discretized time-dependent Stokes extension on two-dimensional Taylor--Hood discretizations.
\end{enumerate}

Section~\ref{sec:problem} specifies the geometry, Stokes problem, and supporting estimates. Section~\ref{sec:main} proves the norm equivalence and its discrete counterpart. Section~\ref{sec:preconditioner} derives the Riesz map and preconditioner; Section~\ref{sec:numerical} presents the experiments; and Section~\ref{sec:conclusion} concludes.

\section{Mathematical Description of the problem}\label{sec:problem}

\paragraph{Asymptotic notation.}
For nonnegative quantities depending on the small scale $\eps$, $f\asymp g$ means that $c g\le f\le C g$ with constants $c,C>0$ independent of $\eps$; $f\sim g$ means that $f/g$ tends to a finite, nonzero limit in the indicated asymptotic regime. In descriptive non-asymptotic statements, $\sim$ denotes an approximate numerical scale. The notation $x\uparrow x_0$ means that $x$ increases towards $x_0$.

\subsection{Geometric Configuration of the Periodic Pillars}

\revision{Following the two-scale framework of Allaire \cite{Allaire1991a,Allaire1991} for periodically perforated domains, we describe the pillar array through a single \emph{model obstacle} placed in a reference cell and replicated over a lattice. Two independent length scales govern the geometry: the inter-pillar \emph{period} $\eps>0$, which fixes the spacing of the lattice, and the \emph{hole size} $a_\eps>0$, which fixes the diameter of each individual pillar. The perforated fluid domain is obtained by removing the rescaled obstacles from a fixed macroscopic container, and all subsequent estimates are made uniform in the ratio $a_\eps/\eps$.}{Following the two-scale framework of Allaire \cite{Allaire1991a,Allaire1991} for periodically perforated domains, we describe the pillar array through one \emph{reference cell} containing one \emph{model obstacle}, replicated over a lattice. The reference cell fixes the local pore geometry; the macroscopic container fixes the device geometry. Two independent length scales govern the construction: the inter-pillar \emph{period} $\eps>0$, which is the size of a physical cell, and the \emph{hole size} $a_\eps>0$, which fixes the diameter of each individual pillar. The perforated fluid domain is obtained by removing these rescaled obstacles from the fixed macroscopic container, and all estimates below are uniform as $\eps\to0$ under the stated separation assumptions.}

\paragraph{Macroscopic container.}
Let the unperforated (``filled'') domain be $\Omega \subset \mathbb{R}^d$ ($d \in \{2, 3\}$), built over a rectangular cross-section $\Omega_{xy} = (0, L_x) \times (0, L_y)$. For $d=2$ the domain is planar, $\Omega = \Omega_{xy}$; for $d=3$ it is the microfluidic channel $\Omega = \Omega_{xy} \times (0, L_z)$ of height $L_z$. The macroscopic dimensions $L_x,L_y,L_z$ are fixed, independent of $\eps$.

\paragraph{Reference cell and model obstacle.}
In the cross-sectional plane we fix the two lattice vectors
\begin{equation}
    \hat{\mathbf{e}}_1 = (1, \delta), \qquad \hat{\mathbf{e}}_2 = (0, 1), \qquad \delta \in [0, 1),
\end{equation}
where $\delta$ is the DLD row-shift fraction ($\delta=0$ recovers a square array). Let $P_0 \subset \mathbb{R}^2$ be the open unit cell centered at the origin and spanned by $\hat{\mathbf{e}}_1,\hat{\mathbf{e}}_2$, and let the model obstacle $T \subset \mathbb{R}^2$ be a closed, connected, bounded set with Lipschitz boundary. Throughout we impose the following standing hypotheses on $(P_0,T)$, uniform in $\eps$:
\begin{itemize}
    \item[\textbf{(G1)}] \emph{Nondegeneracy.} There exist concentric balls and constants $0 < r_0 \le R_0$ with
    \begin{equation}\label{eq:two_ball}
        B(0, r_0) \subset T \subset B(0, R_0) \subset P_0,
    \end{equation}
    so that $T$ is comparable to a disk at scale $a_\eps$ and is compactly contained in its cell. In particular $P_0 \setminus T$ is connected, so the perforated cell carries a uniform fluid neighborhood around the obstacle.
    \item[\textbf{(G2)}] \emph{Dilute separation.} By \eqref{eq:two_ball} the physical pillar $a_\eps T$ is contained in the ball $B(\mathbf 0, a_\eps R_0)$, so $a_\eps R_0$ is its circumscribed radius. We define the \emph{relative pillar size} as the radius-to-period ratio
    \begin{equation}\label{eq:kappa_def}
        \kappa \;:=\; \sup_\eps \frac{a_\eps R_0}{\eps} \;\in\; \Bigl(0,\tfrac12\Bigr),
    \end{equation}
    and assume $\kappa<\tfrac12$. Equivalently, every pillar fits inside a disk of radius $\kappa\eps$ about its node, so adjacent pillars (whose nodes are at distance $\ge\eps$) are separated by pore throats of width at least $(1-2\kappa)\eps$; the bound $\kappa<\tfrac12$ keeps these throats of width $\Theta(\eps)$ and the solid volume fraction away from close packing. For the disk pillars used in the numerical experiments ($T=B(\mathbf0,1)$, $r_0=R_0=1$), the scale $a_\eps$ is exactly the pillar radius and $\kappa$ reduces to the radius-to-pitch ratio reported there.
\end{itemize}

\paragraph{Periodic distribution.}
The lattice nodes in the cross-sectional plane are
\begin{equation}
    \mathbf{x}_{i,j} = \epsilon (i \hat{\mathbf{e}}_1 + j \hat{\mathbf{e}}_2), \qquad (i,j) \in \mathbb{Z}^2 .
\end{equation}
The pillar attached to node $(i,j)$ is the model obstacle scaled to size $a_\eps$ and translated to $\mathbf{x}_{i,j}$; in three dimensions it is extruded across the full channel height as a cylinder of cross-section $a_\eps T$:
\begin{equation}
    T_{\epsilon}^{i,j} =
    \begin{cases}
        \mathbf{x}_{i,j} + a_{\epsilon} T                   & \text{for } d = 2, \\
        (\mathbf{x}_{i,j} + a_{\epsilon} T) \times [0, L_z] & \text{for } d = 3 .
    \end{cases}
\end{equation}

\paragraph{Interior holes and the perforated domain.}
As in Allaire's construction, only holes lying strictly inside the container are retained, so that no pillar is truncated by the lateral walls $\partial\Omega_{xy}$ and the no-slip boundary stays clean. Writing $P_\epsilon^{i,j} := \mathbf{x}_{i,j} + \epsilon P_0$ for the physical cell of node $(i,j)$, the active index set is
\begin{equation}
    \mathcal{I}_{\epsilon} := \bigl\{ (i,j) \in \mathbb{Z}^2 : \overline{P_\epsilon^{i,j}} \subset \Omega_{xy} \bigr\}.
\end{equation}
The effective (perforated) fluid domain is then the container with all interior pillars removed,
\begin{equation}\label{eq:perforated_domain}
    \Omega_m := \Omega \setminus \bigcup_{(i,j) \in \mathcal{I}_{\epsilon}} T_{\epsilon}^{i,j},
\end{equation}
where the subscript $m \propto \epsilon^{-1}$ records the linear pillar density per unit length. By \textbf{(G1)}--\textbf{(G2)} the domain $\Omega_m$ is connected with Lipschitz boundary $\partial \Omega_m = \partial \Omega \cup \bigl(\bigcup_{(i,j)\in\mathcal I_\eps} \partial T_{\epsilon}^{i,j}\bigr)$, and its geometry is controlled by the two parameters $(\eps, a_\eps)$ alone. The proportional-hole regime $a_\eps \sim \eps$, in which $\sigma_\eps \asymp \eps \asymp m^{-1}$, is the setting of our main result.

\subsection{The Stokes Problem and the inf-sup Condition}

We consider the steady, incompressible Stokes equations modeling the creeping flow of a viscous fluid through the effective fluid domain $\Omega_m$:
\begin{equation} \label{eq:stokes_strong}
    \begin{cases}
        -\mu \Delta \mathbf{u} + \nabla p = f & \text{in } \Omega_m,          \\
        \nabla \cdot \mathbf{u} = 0           & \text{in } \Omega_m,          \\
        \mathbf{u} = 0                        & \text{on } \partial \Omega_m,
    \end{cases}
\end{equation}
where $\mathbf{u}$ is the fluid velocity, $p$ is the pressure, $\mu > 0$ is the dynamic viscosity, and $f$ represents the body force. The boundary $\partial \Omega_m = \partial \Omega \cup (\cup_{i,j} \partial T_{\epsilon}^{i,j})$ incorporates the no-slip condition on both the exterior channel walls and the surfaces of all internal pillars.

The viscosity enters \eqref{eq:stokes_strong} only as a scalar multiple of the velocity bilinear form and can be normalized away: dividing the momentum balance by $\mu$ and setting $p\mapsto p/\mu$, $f\mapsto f/\mu$ yields the unit-viscosity Stokes system. We therefore take $\mu=1$ throughout the steady analysis without loss of generality; in particular, the inf-sup condition and the pressure-norm equivalence of Section~\ref{sec:main} are purely geometric and independent of $\mu$. The viscosity re-enters explicitly in the time-discrete extension of Section~\ref{sec:preconditioner}, through the combination $(\tau\mu)^{1/2}$.

To cast \eqref{eq:stokes_strong} into a variational form, we introduce the standard functional spaces. The velocity space is $V(\Omega_m) = H_0^1(\Omega_m)^d$, equipped with the norm $\|\mathbf{u}\|_V = \|\nabla \mathbf{u}\|_{L^2(\Omega_m)}$. The pressure space is $Q(\Omega_m) = L_0^2(\Omega_m) = \{ q \in L^2(\Omega_m) : \int_{\Omega_m} q \, \mathrm{d}x = 0 \}$, equipped with the standard $L^2$-norm. Throughout, we realize spaces ``modulo constants'' as zero-mean subspaces: $L_0^2(D)$ is isometric to the quotient $L^2(D)/\mathbb R$, since $\inf_{c\in\mathbb R}\norm{q-c}_{L^2(D)}$ is attained at the mean $c=\fint_D q$ and equals $\norm{q}_{L^2(D)}$ for $q\in L_0^2(D)$; correspondingly we write $\dot H^1(D):=H^1(D)\cap L_0^2(D)$ for $H^1$ modulo constants, with seminorm $\norm{r}_{\dot H^1(D)}=\norm{\nabla r}_{L^2(D)}$. We use $L_0^2$ (rather than the quotient notation) consistently below.

Multiplying \eqref{eq:stokes_strong} by test functions and integrating by parts (using the no-slip condition to discard boundary terms) yields the mixed weak formulation: find $(\mathbf{u}, p) \in V(\Omega_m) \times Q(\Omega_m)$ such that
\begin{equation} \label{eq:stokes_weak}
    \begin{aligned}
        a(\mathbf{u}, \mathbf{v}) + b(\mathbf{v}, p) &= \langle f, \mathbf{v} \rangle && \text{for all } \mathbf{v} \in V(\Omega_m), \\
        b(\mathbf{u}, q)                              &= 0                            && \text{for all } q \in Q(\Omega_m),
    \end{aligned}
\end{equation}
where, with the normalization $\mu = 1$, the bilinear forms are
\begin{equation} \label{eq:bilinear_forms}
    a(\mathbf{u}, \mathbf{v}) := \int_{\Omega_m} \nabla \mathbf{u} : \nabla \mathbf{v} \, \mathrm{d}x, \qquad
    b(\mathbf{v}, q) := -\int_{\Omega_m} q \, \nabla \cdot \mathbf{v} \, \mathrm{d}x,
\end{equation}
and $\langle f, \mathbf{v} \rangle$ denotes the duality pairing of the body force with the test field. The form $a$ is coercive on $V(\Omega_m)$ and the off-diagonal form $b$ couples velocity and pressure; the well-posedness of \eqref{eq:stokes_weak} thus hinges on the stability of $b$, measured by the inf-sup condition below.

The well-posedness of the Stokes saddle-point problem, and the robust coupling between the velocity and pressure spaces, are governed by the classical Ladyzhenskaya-Babuška-Brezzi (LBB) condition. The continuous LBB constant for the perforated pillar-array domain $\Omega_m$ is defined as:
\begin{equation} \label{eq:infsup_def}
    \beta(\Omega_m) := \inf_{q \in Q(\Omega_m) \setminus \{0\}} \sup_{\mathbf{v} \in V(\Omega_m) \setminus \{0\}} \frac{b(\mathbf{v}, q)}{\|q\|_{L^2(\Omega_m)} \, \|\nabla \mathbf{v}\|_{L^2(\Omega_m)}}.
\end{equation}
\section{Norm Equivalence in the Perforated Domain}\label{sec:main}

\subsection{Pressure-stability norms}

Define
\[
\sigma_\eps := \eps \bigl|\log(a_\eps/\eps)\bigr|^{1/2},
\]
so that when $a_\eps\propto\eps$, $\sigma_\eps = \Theta(\eps)=\Theta(m^{-1})$.
This is the natural perforation scale appearing in the restriction
estimate \cite{Allaire1991,Lu2020}, and it plays the role of the local channel
width in our $K$-functional norm.
The $K$-functional (sum norm) is defined by
\begin{equation}\label{eq:K_norm_OmM}
\norm{q}_{L^2+\sigma_\eps H^1} := \inf_{\tilde q\in H^1(\OmM)} \Bigl( \norm{q-\tilde q}_{L^2(\OmM)}^2 + \sigma_\eps^2\norm{\nabla\tilde q}_{L^2(\OmM)}^2 \Bigr)^{1/2}.
\end{equation}

To estimate the inf-sup condition, we introduce the inf-sup norm on $Q(\OmM)$:
\[
\norm{q}_{*,\OmM} := \sup_{\mathbf{v}\in V(\OmM)\setminus\{\mathbf{0}\}} \frac{\int_{\OmM} q\,\nabla\cdot\mathbf{v}\,dx}{\norm{\nabla\mathbf{v}}_{L^2(\OmM)}}.
\]
The main purpose of this paper is to establish the equivalence of $\norm{q}_{*,\OmM}$ and $\norm{q}_{L^2+\sigma_\eps H^1}$ with constants independent of $\eps$.

\begin{theorem}
\label{thm:main}
\revision{Under the geometric assumptions stated in Section 1, there exist constants $c,C>0$ independent of $\eps$ (and hence independent of the pillar density $m$) such that for all $q\in Q(\OmM)$,}{Under the geometric assumptions of Section~\ref{sec:problem}, in either the planar ($d=2$) or extruded-channel ($d=3$) setting, there exist constants $c,C>0$ independent of $\eps$ (and hence independent of the pillar density $m$) such that for all $q\in Q(\OmM)$,}
\begin{equation}
c\,\norm{q}_{L^2+\sigma_\eps H^1} \le \norm{q}_{*,\OmM} \le C\,\norm{q}_{L^2+\sigma_\eps H^1}.
\end{equation}
\end{theorem}

\paragraph{From homogenization to a pressure norm.}
Theorem~\ref{thm:main} sits at the end of a long arc of ideas that traces back
to Brinkman's 1947 ansatz \cite{Brinkman1947}. Seeking the drag on a dense swarm of particles,
Brinkman wrote down by hand the equation
$-\Delta\mathbf u+\sigma^{-2}\mathbf u+\nabla p=\mathbf f$, an interpolation
between viscous Stokes flow and Darcy's law---a phenomenological guess at the
physics between the dilute and the densely packed regime. Homogenization theory
later explained why the guess is right: as the perforation period $\eps\to0$,
Stokes flow through a domain riddled with holes converges to a Darcy or Brinkman
law, with the friction coefficient $\sigma_\eps^{-2}$ \emph{emerging from the
holes themselves}---the ``strange term coming from nowhere'' of Cioranescu and
Murat \cite{CioranescuMurat1997}, made systematic by Tartar
\cite{Tartar1980,SanchezPalencia1980} and Allaire \cite{Allaire1991a,Allaire1991}
and unified by Lu \cite{Lu2020}. The technical engine of that theory is the
\emph{restriction operator} $R_\eps$: a divergence-preserving map that squeezes a
velocity field from the filled domain into the perforated one, paying an energy
price calibrated precisely to the perforation scale $\sigma_\eps$. It was built
to take the limit---to manufacture the homogenized equation out of the
pore-scale flow.

Our proof runs this machine in reverse. Rather than send $\eps\to0$ to read off
a homogenized coefficient, we dualize $R_\eps$ at \emph{fixed} $\eps$ and read
off a pressure norm. The very scale $\sigma_\eps$ that homogenization deposits
as the Brinkman drag reappears, untouched, as the weight in the $K$-functional
pressure norm $\norm{\cdot}_{L^2+\sigma_\eps H^1}$---and hence as the coefficient
of the preconditioner of Section~\ref{sec:preconditioner}. Where the
homogenization literature records a one-sided estimate on a single Stokes
pressure, we obtain the two-sided equivalence of Theorem~\ref{thm:main}, valid
for every $q\in L_0^2(\OmM)$. The Brinkman interpolation is thus not merely an
analogy for the preconditioner: it is the same object---the homogenized
permeability---read backwards from the pore scale.

\paragraph{Roadmap of the proof.}
The upper bound (Section~\ref{subsec:upper_bound}) is elementary, following from a
triangle inequality and the $\sigma_\eps$-weighted Poincar\'e inequality on
$\OmM$. The lower bound is the heart of the paper and proceeds in two steps:
Section~\ref{subsec:restriction_extension} dualizes the restriction operator
$R_\eps : H_0^1(\Omega)^d\to H_0^1(\OmM)^d$ through de Rham's theorem into a
pressure extension $\mathcal E_\eps : L_0^2(\OmM)\to L_0^2(\Omega)$, and
Section~\ref{sec:lower_bound_proof} makes this quantitative via a fixed-domain
weighted Ne\v{c}as inequality on $\Omega$ at the $K$-functional scale. The
continuous-to-discrete transport via a Fortin operator is carried out in
Section~\ref{subsec:fortin}.

\subsection{Upper bound}\label{subsec:upper_bound}

\begin{proof}[Proof of Theorem~\ref{thm:main}, upper bound]
Let $q\in Q(\OmM)$ be arbitrary and choose any $\tilde q\in H^1(\OmM)$. By the triangle inequality for the inf-sup norm,
\[
\norm{q}_{*,\OmM} \le \norm{q-\tilde q}_{*,\OmM} + \norm{\tilde q}_{*,\OmM}.
\]

For the first term, we evaluate the inf-sup supremum directly. For any $\mathbf{v}\in V(\OmM)$ with $\mathbf{v} \neq \mathbf{0}$, we apply the Cauchy-Schwarz inequality:
\[
\int_{\OmM}(q-\tilde q)\nabla\cdot\mathbf{v}\,dx
\le \norm{q-\tilde q}_{L^2(\OmM)} \norm{\nabla\cdot\mathbf{v}}_{L^2(\OmM)}
\le \sqrt{d} \norm{q-\tilde q}_{L^2(\OmM)} \norm{\nabla\mathbf{v}}_{L^2(\OmM)},
\]
which immediately yields
\[
\norm{q-\tilde q}_{*,\OmM} \le \sqrt{d} \norm{q-\tilde q}_{L^2(\OmM)}.
\]

For the second term, since $\tilde{q} \in H^1(\OmM)$ and $\mathbf{v} \in H_0^1(\OmM)^d$, we integrate by parts and apply the $\sigma_\eps$-weighted Poincar\'e inequality on $\Omega_m$, $\norm{\mathbf{v}}_{L^2(\OmM)} \le C_P\,\sigma_\eps\,\norm{\nabla\mathbf{v}}_{L^2(\OmM)}$ with $C_P$ independent of $\eps$ \cite{infsup_perforation}:
\[
\int_{\OmM} \tilde{q}\,\nabla\cdot\mathbf{v}\,dx
= -\int_{\OmM} \nabla\tilde{q} \cdot \mathbf{v}\,dx
\le \norm{\sigma_\eps\nabla\tilde{q}}_{L^2(\OmM)} \cdot \sigma_\eps^{-1}\norm{\mathbf{v}}_{L^2(\OmM)}
\le C_P \norm{\sigma_\eps\nabla\tilde{q}}_{L^2(\OmM)} \norm{\nabla\mathbf{v}}_{L^2(\OmM)}.
\]
Taking the supremum over all non-zero $\mathbf{v} \in V(\OmM)$ gives
\[
\norm{\tilde q}_{*,\OmM} \le C_P\,\norm{\sigma_\eps\nabla\tilde{q}}_{L^2(\OmM)}.
\]

Combining these estimates, we obtain
\[
\norm{q}_{*,\OmM} \le \sqrt{d} \norm{q-\tilde q}_{L^2(\OmM)} + C_P\,\norm{\sigma_\eps\nabla\tilde q}_{L^2(\OmM)}.
\]
By the Cauchy-Schwarz inequality applied to the vector $(\sqrt{d}, C_P)$, this is bounded by
\[
\norm{q}_{*,\OmM} \le \sqrt{d + C_P^2} \Bigl( \norm{q-\tilde q}_{L^2(\OmM)}^2 + \norm{\sigma_\eps\nabla\tilde q}_{L^2(\OmM)}^2 \Bigr)^{1/2}.
\]
Since this inequality holds for any $\tilde{q} \in H^1(\OmM)$, we take the infimum over all such $\tilde{q}$ to conclude
\[
\norm{q}_{*,\OmM} \le C_U \norm{q}_{L^2+\sigma_\eps H^1},
\]
where $C_U := \sqrt{d + C_P^2}$ is independent of $\eps$. This establishes the upper bound.
\end{proof}

\subsection{Restriction operator and the pressure extension}\label{subsec:restriction_extension}

The lower bound rests on a \emph{pressure extension} from $\OmM$ to the
filled domain $\Omega$, built by dualizing the restriction operator
introduced by Allaire \cite{Allaire1991}.
The construction in this subsection is purely existential: we recall the
restriction operator $R_\eps$, package its action against pressures into a
linear functional $\ell_q$ on $H_0^1(\Omega)^d$, and apply de Rham's theorem
to produce a pressure $\mathcal E_\eps q\in L^2(\Omega)/\mathbb R$ whose
gradient realizes that functional. The quantitative ingredient that turns
this duality into a $K$-functional bound---a fixed-domain weighted
Ne\v{c}as inequality on $\Omega$---is the engine of the lower bound and is
postponed to Section~\ref{sec:lower_bound_proof}, where it is introduced
exactly at the point of use.

\paragraph{The $K$-functional seminorm on a general domain.}
The proof compares pressures across two domains, so we extend the
$K$-functional norm \eqref{eq:K_norm_OmM} from $\OmM$ to an arbitrary open set
$D\subset\mathbb R^d$: for $r\in L^2(D)$,
\begin{equation}
\label{eq:K_sigma_norm_def}
    \norm{r}_{K_{\sigma_\eps}(D)}
    :=
    \inf_{\phi\in H^1(D)}
    \left(
        \norm{r-\phi}_{L^2(D)}^2
        +
        \sigma_\eps^2
        \norm{\nabla \phi}_{L^2(D)}^2
    \right)^{1/2}
    =
    \norm{r}_{L^2+\sigma_\eps H^1(D)} ,
\end{equation}
a seminorm on $L^2(D)$ and a norm on $L_0^2(D)$. Taking $D=\OmM$ recovers
\eqref{eq:K_norm_OmM}; the lower bound also uses it on the filled domain
$D=\Omega$.

\paragraph{The restriction operator.}
The only external ingredient we need is the following restriction estimate, due
to Allaire \cite{Allaire1991a,Allaire1991} in the periodic proportional-hole regime and
revisited by Lu \cite{Lu2020}.

\begin{proposition}[Restriction operator; \cite{Allaire1991,Lu2020}]
\label{prop:restriction_operator_lu_allaire}
Under the geometric assumptions of Section~\ref{sec:problem}, there exists a
linear operator $R_\eps: H_0^1(\Omega)^d\to H_0^1(\OmM)^d$ with the
following properties.
\begin{enumerate}[label=(\arabic*)]
    \item If $\mathbf v\in H_0^1(\OmM)^d$ and $\widetilde{\mathbf v}$
    denotes its zero extension to $\Omega$, then
    $R_\eps \widetilde{\mathbf v} = \mathbf v$ in $\OmM$.

    \item If $\mathbf \Phi\in H_0^1(\Omega)^d$ satisfies
    $\nabla\cdot \mathbf \Phi=0$ in $\Omega$, then
    $\nabla\cdot (R_\eps \mathbf \Phi)=0$ in $\OmM$.

    \item \revision{There exists a constant $C_R>0$, independent of $\eps$, such that}{There exists a constant $C_R>0$, independent of $\eps$ but dependent on the fixed reference-cell geometry (in particular on the separation from close packing), such that}
    \begin{equation}
    \label{eq:restriction_operator_estimate}
        \norm{\nabla R_\eps \mathbf \Phi}_{L^2(\OmM)}
        \le
        C_R
        \Bigl(
            \norm{\nabla \mathbf \Phi}_{L^2(\Omega)}
         +
            \sigma_\eps^{-1}
            \norm{\mathbf \Phi}_{L^2(\Omega)}
        \Bigr)
    \end{equation}
    for all $\mathbf \Phi\in H_0^1(\Omega)^d$.
\end{enumerate}
\end{proposition}

In dimension $d=2$ the optimal scale from the explicit cellwise construction of
\cite{Allaire1991} is $\sigma_\eps=\eps|\log(a_\eps/\eps)|^{-1/2}$; for
$a_\eps\sim\eps$ this reduces to $\sigma_\eps\asymp\eps$, which is the regime
considered here. We refer to \cite{Allaire1991,Lu2020} for the construction and \eqref{eq:restriction_operator_estimate}.

\begin{remark}[Scope of the cellwise construction: why close packing is excluded]
\label{rmk:close_packing}
Both the weighted Poincar\'e inequality on $\OmM$ \cite{infsup_perforation}
and the restriction estimate \eqref{eq:restriction_operator_estimate} rest on
the same cellwise mechanism: in each cell the velocity is corrected on the
annular fluid layer between the obstacle and a ball inscribed in the cell. By
the sandwiching assumption $B(0,r_0)\subset T\subset B(0,R_0)\subset P_0$, the
obstacle is pinched between two balls of scale $a_\eps$, while the cell $P_0$ is
controlled from inside by a single ball of scale $\eps$; the correction then
lives on an annulus of radius ratio $\asymp a_\eps/\eps$, and the relevant
constant is the two-dimensional logarithmic (conformal) capacity of that
annulus, which is precisely what produces the factor $|\log(a_\eps/\eps)|$ in
$\sigma_\eps$.

This ball-in-cell picture is faithful in exactly two regimes: vanishing holes
$a_\eps/\eps\to0$, where the annulus is asymptotically a true annulus and the
logarithmic capacity is sharp; and proportional holes with $\kappa$ bounded
away from $1/2$, where the inscribed ball $B(0,R_0)$ and the square cell $P_0$
are comparable up to a $\kappa$-dependent constant. It degrades in the
close-packing limit $\kappa\uparrow1/2$: the fluid between adjacent pillars then
concentrates in thin throats of width $(1-2\kappa)\eps$ at the cell edges and in
the cell corners, a geometry that no single inscribed ball can resolve
(Figure~\ref{fig:ball_vs_cell}).
Replacing the square cell by a ball of radius $\asymp\eps$ discards exactly the
corner and throat fluid that carries the singular behaviour, so the restriction
constant $C_R$ in \eqref{eq:restriction_operator_estimate} is no longer
controlled by this cellwise mechanism and degenerates as $\kappa\uparrow1/2$.
This is why the logarithmic factor in
$\sigma_\eps$ ceases to be the relevant small parameter in this regime: it
measures an annular capacity, whereas the true degeneration is a thin-throat
effect the construction cannot see. The tools used here are therefore tailored
to small and proportionally sized holes with $\kappa$ bounded away from $1/2$;
the dependence of the restriction constant on the hole-size regime is analysed
in detail in \cite{Lu2020}. The close-packing regime requires genuinely
different, throat-resolving estimates and lies outside the scope of the present
paper.
\end{remark}

\begin{figure}[t]
    \centering
    \begin{subfigure}[b]{0.46\textwidth}
        \centering
        \begin{tikzpicture}[scale=1.9, line join=round, >=stealth]
          \fill[red!16, even odd rule] (-1,-1) rectangle (1,1) (0,0) circle (1);
          \fill[blue!12, even odd rule] (0,0) circle (1) (0,0) circle (0.36);
          \draw[very thick] (-1,-1) rectangle (1,1);
          \draw[dashed, thick] (0,0) circle (1);
          \fill[gray!55] (0,0) circle (0.36);
          \draw[thick] (0,0) circle (0.36);
          \draw[<->, thick] (0.36,0) -- (1,0);
          \node[fill=white, inner sep=1pt] at (0.68,0) {\scriptsize $g$};
        \end{tikzpicture}
        \caption{Dilute / proportional holes ($\kappa$ bounded away from $1/2$):
        ball $\approx$ cell.}
        \label{fig:ball_small}
    \end{subfigure}
    \hfill
    \begin{subfigure}[b]{0.46\textwidth}
        \centering
        \begin{tikzpicture}[scale=1.9, line join=round, >=stealth]
          \fill[red!16, even odd rule] (-1,-1) rectangle (1,1) (0,0) circle (1);
          \fill[blue!12, even odd rule] (0,0) circle (1) (0,0) circle (0.85);
          \draw[very thick] (-1,-1) rectangle (1,1);
          \draw[dashed, thick] (0,0) circle (1);
          \fill[gray!55] (0,0) circle (0.85);
          \draw[thick] (0,0) circle (0.85);
          \draw[<->, thick] (0.85,0) -- (1,0);
          \node[fill=white, inner sep=1pt] at (1.16,0) {\scriptsize $g\!\to\!0$};
        \end{tikzpicture}
        \caption{Close packing ($\kappa\uparrow1/2$): ball $\not\approx$ cell,
        thin annulus, large lost fluid.}
        \label{fig:ball_large}
    \end{subfigure}
    \caption{Why the inscribed-ball construction is faithful for dilute and
    proportional holes but fails at close packing
    (Remark~\ref{rmk:close_packing}). In each cell the restriction/Poincar\'e
    correction is built on the \emph{annulus} (blue) between the obstacle (gray,
    scale $a_\eps$) and the \emph{ball inscribed in the square cell} (dashed,
    scale $\eps$); the four corner regions (red) are the fluid that the ball
    approximation of the cell discards, and $g=\tfrac12(1-2\kappa)\eps$ is the
    gap from the obstacle to the cell wall (half a throat).
    \textbf{(a)} When $\kappa$ is bounded away from $1/2$ the annulus is fat and
    the discarded corners are a lower-order perturbation, so the ball is
    comparable to the cell and the annular (logarithmic) capacity controls the
    constant. \textbf{(b)} As $\kappa\uparrow1/2$ the obstacle nearly fills the
    inscribed ball: the controlled annulus collapses to a razor-thin ring while
    the discarded corners and the edge throats of width $(1-2\kappa)\eps$ carry
    the remaining fluid. The ball can no longer see the cell, $C_R$
    degenerates as $\kappa\uparrow1/2$, and the frozen factor $|\log(a_\eps/\eps)|$ no
    longer measures the degeneration.}
    \label{fig:ball_vs_cell}
\end{figure}
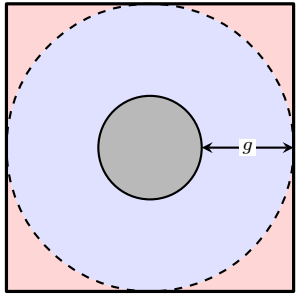
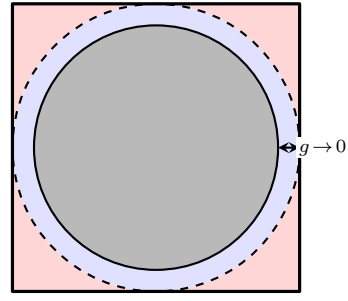

\paragraph{The pressure extension via de Rham's theorem.}
With the restriction operator in hand, we now build the pressure extension
by lifting a pressure $q$ from $\OmM$ to a pressure $Q$ on the full domain
$\Omega$, in a way that turns the inf-sup norm $\norm{q}_{*,\OmM}$ into a
controllable dual quantity on $\Omega$. The lifting is carried out by
duality through $R_\eps$: rather than constructing $Q$ pointwise, we
identify it via its action against gradients of test velocities. The
existence of such a $Q$ rests on a classical statement of de Rham, which
we record in the form used below.

\begin{lemma}[de Rham's theorem on a bounded Lipschitz domain;
{see \cite{Girault1986}}]
\label{lem:de_rham}
Let $U\subset\mathbb R^d$ be a bounded Lipschitz domain and let
$\ell\in H^{-1}(U)^d := \bigl(H_0^1(U)^d\bigr)'$. The following are
equivalent:
\begin{enumerate}[label=(\arabic*)]
    \item $\ell(\mathbf \Phi)=0$ for every
    $\mathbf \Phi\in H_0^1(U)^d$ with $\nabla\!\cdot\!\mathbf \Phi=0$ in $U$;
    \item there exists $Q\in L^2(U)$, unique up to an additive constant, with
    \begin{equation}
    \label{eq:de_rham_representation}
        \ell(\mathbf \Phi)
        \;=\;
        \bigl\langle \nabla Q,\,\mathbf \Phi\bigr\rangle_U
        \;=\;
        -\int_U Q\,\nabla\!\cdot\!\mathbf \Phi\,dx,
        \qquad
        \forall\,\mathbf \Phi\in H_0^1(U)^d.
    \end{equation}
\end{enumerate}
\end{lemma}

\begin{proof}
For completeness we include a simple proof tailored to our setting.
The implication (2)$\Rightarrow$(1) is immediate, since for any
divergence-free $\mathbf\Phi$,
\[
    \langle\nabla Q,\mathbf\Phi\rangle_U
    =-\int_U Q\,\nabla\!\cdot\!\mathbf\Phi\,dx=0 .
\]
For (1)$\Rightarrow$(2), consider the divergence operator and its adjoint,
\[
    \operatorname{div}:H_0^1(U)^d\to L_0^2(U),
    \qquad
    (\operatorname{div})^*=-\nabla:L_0^2(U)\to H^{-1}(U)^d ,
\]
whose kernel
$\ker(\operatorname{div})=\{\mathbf\Phi\in H_0^1(U)^d:\nabla\!\cdot\!\mathbf\Phi=0\}$
is the closed subspace of divergence-free fields. Condition~(1) says exactly
that $\ell$ annihilates $\ker(\operatorname{div})$, i.e.
$\ell\in\ker(\operatorname{div})^\perp$. On a bounded Lipschitz domain the
Bogovskii operator \cite{Galdi2011} furnishes a bounded
right inverse of $\operatorname{div}$; hence $\operatorname{div}$ is surjective
with closed range, equivalently the Ne\v{c}as inequality
\[
    \norm{Q}_{L_0^2(U)}\le C\,\norm{\nabla Q}_{H^{-1}(U)^d}
\]
holds. By the closed-range theorem the adjoint $\nabla$ then also has closed
range, and
\[
    \operatorname{Range}(\nabla)=\ker(\operatorname{div})^\perp .
\]
Therefore $\ell\in\ker(\operatorname{div})^\perp$ is represented as
$\ell=\nabla Q$ for some $Q\in L_0^2(U)$, which is exactly
\eqref{eq:de_rham_representation}; uniqueness up to a constant follows from
$\nabla Q=0\Rightarrow Q$ constant. On the filled box $U=\Omega$ used below the
surjectivity of $\operatorname{div}$ (equivalently the inf-sup/Ne\v{c}as input)
is classical, with a constant depending only on the dimension and the aspect
ratio of $\Omega$, so the construction carries no $\eps$-dependence.
\end{proof}

To use Lemma~\ref{lem:de_rham} we choose $U=\Omega$ (the filled domain) and we
build, from each $q\in L_0^2(\OmM)$, a functional on $H_0^1(\Omega)^d$ that
encodes how a genuine extension of $q$ \emph{ought} to pair against test
velocities. The natural choice is
\begin{equation}
\label{eq:pressure_extension_functional}
    \ell_q(\mathbf \Phi)
    \;:=\;
    -\int_{\OmM} q\,\nabla\!\cdot\!(R_\eps\mathbf \Phi)\,dx,
    \qquad
    \mathbf \Phi\in H_0^1(\Omega)^d.
\end{equation}
The reason for this definition is the following. If a function
$Q\in L^2(\Omega)$ were to extend $q$ from $\OmM$ in the
integration-by-parts sense, then for every $\mathbf v\in H_0^1(\OmM)^d$ one
would have
\(
    \langle\nabla Q,\widetilde{\mathbf v}\rangle_\Omega
    =-\int_\Omega Q\,\nabla\!\cdot\!\widetilde{\mathbf v}\,dx
    =-\int_{\OmM} q\,\nabla\!\cdot\!\mathbf v\,dx,
\)
where $\widetilde{\mathbf v}$ is the zero extension of $\mathbf v$ to
$\Omega$. The restriction operator
$R_\eps:H_0^1(\Omega)^d\to H_0^1(\OmM)^d$ provided by
Proposition~\ref{prop:restriction_operator_lu_allaire} is precisely the
device that turns a generic test velocity $\mathbf \Phi$ on $\Omega$ into an
admissible velocity on $\OmM$, so
\eqref{eq:pressure_extension_functional} is the pairing one would expect from
a true extension of $q$.

Two properties of $\ell_q$ make Lemma~\ref{lem:de_rham} applicable on
$U=\Omega$:

\smallskip
\noindent\textbf{(i) $\ell_q\in H^{-1}(\Omega)^d$, with a quantitative bound at the
$K$-functional scale.} The inf-sup definition of $\norm{\cdot}_{*,\OmM}$ and
the restriction estimate \eqref{eq:restriction_operator_estimate} give, for
every $\mathbf \Phi\in H_0^1(\Omega)^d$,
\begin{equation}
\label{eq:ell_q_bound}
    |\ell_q(\mathbf \Phi)|
    \;\le\;
    \norm{q}_{*,\OmM}\,\norm{\nabla R_\eps\mathbf \Phi}_{L^2(\OmM)}
    \;\le\;
    C_R\,\norm{q}_{*,\OmM}
    \Bigl(
        \norm{\nabla \mathbf \Phi}_{L^2(\Omega)}
        +\sigma_\eps^{-1}\norm{\mathbf \Phi}_{L^2(\Omega)}
    \Bigr).
\end{equation}

\smallskip
\noindent\textbf{(ii) $\ell_q$ vanishes on divergence-free test fields.}
If $\mathbf \Phi\in H_0^1(\Omega)^d$ satisfies $\nabla\!\cdot\!\mathbf \Phi=0$
in $\Omega$, then property~(2) of
Proposition~\ref{prop:restriction_operator_lu_allaire} gives
$\nabla\!\cdot\!(R_\eps\mathbf \Phi)=0$ in $\OmM$, so
\eqref{eq:pressure_extension_functional} yields
$\ell_q(\mathbf \Phi)=0$. This is condition~(1) of Lemma~\ref{lem:de_rham}.

\smallskip
Lemma~\ref{lem:de_rham} therefore furnishes a pressure
$Q=\mathcal E_\eps q\in L^2(\Omega)$, unique up to additive constants, with
\begin{equation}
\label{eq:pressure_extension_distribution}
    \bigl\langle \nabla Q,\mathbf \Phi\bigr\rangle_\Omega
    \;=\;\ell_q(\mathbf \Phi),
    \qquad
    \forall \mathbf \Phi\in H_0^1(\Omega)^d.
\end{equation}
We call $\mathcal E_\eps q$ the \emph{(dual) pressure extension} of $q$.
By construction, the map $\mathcal E_\eps : L_0^2(\OmM)\to L_0^2(\Omega)$
is linear.


\begin{lemma}[Consistency of the pressure extension]
\label{lem:pressure_extension_consistency}
For every $q\in L_0^2(\OmM)$, the pressure extension
$Q=\mathcal E_\eps q$ satisfies $Q=q+c_\eps$ in $\OmM$ for some
$c_\eps\in\mathbb R$. Normalizing by $\int_{\OmM}Q\,dx=0$ gives
$Q=q$ in $\OmM$.
\end{lemma}

\begin{proof}
By Lemma~\ref{lem:de_rham}, $Q\in L^2(\Omega)$ up to an additive constant, so
$Q|_{\OmM}\in L^2(\OmM)$. Let
$\varphi\in L_0^2(\OmM)$ be arbitrary. Since $\OmM$ is a bounded
Lipschitz open set, the Bogovskii operator on $\OmM$
\cite{Galdi2011,infsup_perforation} provides
$\mathbf b\in H_0^1(\OmM)^d$ with $\nabla\cdot\mathbf b=\varphi$ in
$\OmM$; the $\eps$-dependence of the Bogovskii constant is irrelevant
here because $\mathbf b$ only plays the role of a test function. Let
$\widetilde{\mathbf b}\in H_0^1(\Omega)^d$ be the zero extension of
$\mathbf b$ to $\Omega$; by property~(1) of
Proposition~\ref{prop:restriction_operator_lu_allaire},
$R_\eps\widetilde{\mathbf b}=\mathbf b$ in $\OmM$. Using the definition
of $Q$,
\[
\begin{aligned}
    \int_{\OmM} Q\,\varphi\,dx
    &=\int_\Omega Q\,\nabla\cdot\widetilde{\mathbf b}\,dx
    =-\bigl\langle\nabla Q,\widetilde{\mathbf b}\bigr\rangle_\Omega
    =-\ell_q(\widetilde{\mathbf b})\\
    &=\int_{\OmM}q\,\nabla\cdot(R_\eps\widetilde{\mathbf b})\,dx
    =\int_{\OmM}q\,\nabla\cdot\mathbf b\,dx
    =\int_{\OmM}q\,\varphi\,dx.
\end{aligned}
\]
Hence $\int_{\OmM}(Q-q)\varphi\,dx=0$ for all
$\varphi\in L_0^2(\OmM)$, so $Q-q$ is constant on $\OmM$. Fixing
the constant by $\int_{\OmM}Q\,dx=0$ gives $Q=q$ in $\OmM$.
\end{proof}

\subsection{Proof of the lower bound}\label{sec:lower_bound_proof}

The pressure extension $\mathcal E_\eps q$ constructed in
Section~\ref{subsec:restriction_extension} so far satisfies only the
qualitative identity \eqref{eq:pressure_extension_distribution}; we now turn
it into a quantitative bound at the $K$-functional scale. The right tool is
a fixed-domain form of the Ne\v{c}as inequality on $\Omega$.

\begin{lemma}[Weighted Ne\v{c}as inequality]
\label{lem:weighted_necas_K}
Let $\Omega\subset\mathbb R^d$ be a fixed bounded Lipschitz domain. There exists
$C_N>0$, depending only on $\Omega$ and $d$, such that for every $0<\sigma\le 1$
and every $Q\in L^2(\Omega)$ (understood modulo additive constants),
\begin{equation}
\label{eq:weighted_necas_K}
    \norm{Q}_{K_\sigma(\Omega)}
    \le
    C_N
    \sup_{\mathbf \Phi\in H_0^1(\Omega)^d\setminus\{\mathbf 0\}}
    \frac{
        \bigl|\bigl\langle \nabla Q,\mathbf \Phi\bigr\rangle_\Omega\bigr|
    }{
        \norm{\nabla \mathbf \Phi}_{L^2(\Omega)}
        +\sigma^{-1}\norm{\mathbf \Phi}_{L^2(\Omega)}
    }.
\end{equation}
\end{lemma}

\begin{proof}
Recall from Section~\ref{sec:problem} the zero-mean spaces $L_0^2(\Omega)$ and
$\dot H^1(\Omega)=H^1(\Omega)\cap L_0^2(\Omega)$, with
$\norm{R}_{L_0^2}=\norm{R}_{L^2}$ and
$\norm{R}_{\dot H^1}=\norm{\nabla R}_{L^2}$. The classical Ne\v{c}as inequality
on $\Omega$ \cite{Girault1986} gives
\begin{equation}
\label{eq:necas_classic}
    \norm{R}_{L_0^2(\Omega)}
    \le
    C_{\mathrm{N}}^{(0)}\,\norm{\nabla R}_{H^{-1}(\Omega)^d},
    \qquad R\in L_0^2(\Omega),
\end{equation}
while the endpoint identity $\norm{R}_{\dot H^1(\Omega)}=\norm{\nabla R}_{L^2(\Omega)^d}$ is immediate.
The gradient map $G:R\mapsto\nabla R$ is therefore boundedly invertible on its
range at the two endpoints
\[
    G:L_0^2(\Omega)\to H^{-1}(\Omega)^d
    \ \text{with norm $C_{\mathrm{N}}^{(0)}$},
    \qquad
    G:\dot H^1(\Omega)\to L^2(\Omega)^d
    \ \text{with norm $1$}.
\]
By the real interpolation $K$-method applied to both sides of $G$, using the
$(0,1)$-parameter
$\theta=1/2$ with weight $\sigma$ (so that
$(X_0,X_1)_{K,\sigma}=X_0+\sigma X_1$ with the $K$-norm \eqref{eq:K_sigma_norm_def}),
we obtain from \cite{berghLofstrom1976}
\begin{equation}
\label{eq:K_functional_interpolation_step}
    \norm{Q}_{L_0^2(\Omega)+\sigma\,\dot H^1(\Omega)}
    \le
    \max\bigl(C_{\mathrm{N}}^{(0)},1\bigr)\,
    \norm{\nabla Q}_{H^{-1}(\Omega)^d+\sigma L^2(\Omega)^d}
    \quad\forall\,Q\in L_0^2(\Omega).
\end{equation}
The left-hand side of \eqref{eq:K_functional_interpolation_step} is exactly
$\norm{Q}_{K_\sigma(\Omega)}$ modulo constants. For the right-hand side, the
classical sum-intersection duality
\cite{berghLofstrom1976} gives
\[
    H^{-1}(\Omega)^d+\sigma L^2(\Omega)^d
    =
    \bigl(H_0^1(\Omega)^d\cap \sigma^{-1}L^2(\Omega)^d\bigr)'
\]
with equivalent norms, so that
\begin{equation}
\label{eq:necas_dual_step}
    \norm{\nabla Q}_{H^{-1}+\sigma L^2}
    =
    \sup_{\mathbf \Phi\in H_0^1(\Omega)^d\setminus\{\mathbf 0\}}
    \frac{
        \bigl|\bigl\langle \nabla Q,\mathbf \Phi\bigr\rangle_\Omega\bigr|
    }{
        \norm{\nabla \mathbf \Phi}_{L^2(\Omega)}
        +\sigma^{-1}\norm{\mathbf \Phi}_{L^2(\Omega)}
    }.
\end{equation}
Combining \eqref{eq:K_functional_interpolation_step} and
\eqref{eq:necas_dual_step} yields \eqref{eq:weighted_necas_K} with
$C_N=\max(C_{\mathrm{N}}^{(0)},1)$.
\end{proof}

\begin{remark}[Brinkman interpretation]
\label{rem:brinkman}
The weighted Ne\v{c}as inequality \eqref{eq:weighted_necas_K} can be read as
the dual side of the inf-sup condition for the Brinkman problem
\(
    -\mu\Delta\mathbf u + \mu\sigma^{-2}\mathbf u + \nabla p = \mathbf f
\)
in $\mathbf{u}|_{\partial\Omega}=\mathbf 0$ on the full domain $\Omega$,
at the homogenized permeability $\sigma^2$. The combination
$\norm{\nabla\mathbf\Phi}_{L^2(\Omega)}+\sigma^{-1}\norm{\mathbf\Phi}_{L^2(\Omega)}$
appearing in the denominator is precisely the energy norm of the Brinkman
velocity space (up to a $\mu$-rescaling). The pressure extension
$\mathcal E_\eps$ of Section~\ref{subsec:restriction_extension} is what
converts a Stokes pressure on the perforated domain $\OmM$ into a Brinkman
pressure on the filled domain $\Omega$ at scale $\sigma_\eps$, so that the
perforation-uniform inf-sup of Stokes on $\OmM$ is captured by a
fixed-domain estimate on $\Omega$.
\end{remark}

Combining the weighted Ne\v{c}as inequality with the $K$-functional dual
bound \eqref{eq:ell_q_bound} on $\ell_q$ now produces the quantitative
control of the extension that drives the lower bound.

\begin{lemma}[Boundedness of the pressure extension]
\label{lem:pressure_extension_K_bound}
There exists $C>0$, independent of $\eps$, such that
\begin{equation}
\label{eq:pressure_extension_K_bound}
    \norm{\mathcal E_\eps q}_{K_{\sigma_\eps}(\Omega)}
    \le
    C\,\norm{q}_{*,\OmM},
    \qquad
    \forall q\in L_0^2(\OmM).
\end{equation}
\end{lemma}

\begin{proof}
Let $Q=\mathcal E_\eps q$. Applying Lemma~\ref{lem:weighted_necas_K} with
$\sigma=\sigma_\eps$ and using \eqref{eq:pressure_extension_distribution},
\[
    \norm{Q}_{K_{\sigma_\eps}(\Omega)}
    \le
    C_N
    \sup_{\mathbf \Phi\ne\mathbf 0}
    \frac{|\ell_q(\mathbf \Phi)|}
         {\norm{\nabla\mathbf \Phi}_{L^2(\Omega)}+\sigma_\eps^{-1}\norm{\mathbf \Phi}_{L^2(\Omega)}}
    \le
    C_N C_R\,\norm{q}_{*,\OmM},
\]
where the second inequality is \eqref{eq:ell_q_bound}.
\end{proof}

It remains to push the bound back from $\Omega$ to $\OmM$.

\begin{proof}[Proof of Theorem~\ref{thm:main}, lower bound]
Let $q\in L_0^2(\OmM)=Q(\OmM)$ and let $Q=\mathcal E_\eps q$ be the
pressure extension, normalized by
$\int_{\OmM}Q\,dx=0$ so that $Q=q$ in $\OmM$ by
Lemma~\ref{lem:pressure_extension_consistency}.

For any $\Phi\in H^1(\Omega)$, define the competitor
$\phi:=\Phi|_{\OmM}\in H^1(\OmM)$. Since $Q=q$ on $\OmM$
and restriction reduces the $L^2$ norm,
\[
    \norm{q-\phi}_{L^2(\OmM)}^2
    +\sigma_\eps^2\norm{\nabla\phi}_{L^2(\OmM)}^2
    \le
    \norm{Q-\Phi}_{L^2(\Omega)}^2
    +\sigma_\eps^2\norm{\nabla\Phi}_{L^2(\Omega)}^2.
\]
Taking the infimum over $\Phi\in H^1(\Omega)$,
\[
    \norm{q}_{K_{\sigma_\eps}(\OmM)}
    \le
    \norm{Q}_{K_{\sigma_\eps}(\Omega)}
    \le
    C_N C_R\,\norm{q}_{*,\OmM},
\]
where the second inequality is Lemma~\ref{lem:pressure_extension_K_bound}.
Recalling $\norm{\cdot}_{K_{\sigma_\eps}(\OmM)}=\norm{\cdot}_{L^2+\sigma_\eps H^1(\OmM)}$,
we conclude
\[
    \norm{q}_{L^2+\sigma_\eps H^1(\OmM)}
    \le C_L\,\norm{q}_{*,\OmM},
    \qquad
    C_L:=C_N C_R.
\]
This completes the proof of the theorem.
\end{proof}

\subsection{From the continuous to the discrete norm equivalence}\label{subsec:fortin}

Theorem~\ref{thm:main} is a statement about the continuous norms $\norm{\cdot}_{*,\OmM}$ and $\norm{\cdot}_{L^2+\sigma_\eps H^1}$. To transport it to a conforming finite element subspace $V_h\times Q_h\subset V(\OmM)\times Q(\OmM)$, we impose the following uniform Fortin condition. For classical inf-sup stable pairs (Taylor--Hood $\mathbb P_k$--$\mathbb P_{k-1}$, MINI, the Bernardi--Raugel element, \dots), this is the standard Fortin property; on the perforated pillar-array geometry it is obtained by local quasi-interpolation, with constants depending on pillar shape regularity but not on $\eps$ or $h$ (see, e.g., \cite{boffi2013mixed,sande2025robust}).

\begin{assumption}[Uniform Fortin operator]\label{ass:fortin}
For every $\eps$ and $h$, there is a linear operator
\[
    \Pi_h^{F}:V(\OmM)\longrightarrow V_h
\]
such that, for every $q_h\in Q_h$ and $\mathbf v\in V(\OmM)$,
\begin{equation}\label{eq:fortin}
\int_{\OmM}q_h\,\nabla\cdot(\Pi_h^F\mathbf v)\,dx
 = \int_{\OmM}q_h\,\nabla\cdot\mathbf v\,dx,
\qquad
\norm{\nabla \Pi_h^F\mathbf v}_{L^2(\OmM)}\le C_F\,\norm{\nabla\mathbf v}_{L^2(\OmM)},
\end{equation}
where $C_F$ is independent of $\eps$ and $h$.
\end{assumption}

\begin{corollary}[Discrete norm equivalence]
\label{cor:discrete_equiv}
Let $V_h\times Q_h\subset V(\OmM)\times Q(\OmM)$ be a conforming finite element pair satisfying Assumption~\ref{ass:fortin}. Define the discrete inf-sup norm
\[
\norm{q_h}_{*,\OmM,h}:=\sup_{\mathbf v_h\in V_h\setminus\{\mathbf 0\}}
\frac{\int_{\OmM}q_h\,\nabla\cdot\mathbf v_h\,dx}{\norm{\nabla\mathbf v_h}_{L^2(\OmM)}}.
\]
Then there exist constants $c_*,C_*>0$ independent of $\eps$ and $h$ such that
\begin{equation}
\label{eq:discrete_equiv}
c_*\,\norm{q_h}_{L^2+\sigma_\eps H^1(\OmM)}
\;\le\;
\norm{q_h}_{*,\OmM,h}
\;\le\;
C_*\,\norm{q_h}_{L^2+\sigma_\eps H^1(\OmM)},
\qquad\forall q_h\in Q_h.
\end{equation}
\end{corollary}

\begin{proof}
The upper bound follows by restricting the supremum in the continuous inf-sup norm to the discrete test space: $\norm{q_h}_{*,\OmM,h}\le\norm{q_h}_{*,\OmM}\le C\,\norm{q_h}_{L^2+\sigma_\eps H^1}$ by Theorem~\ref{thm:main}. For the lower bound, take any $\mathbf v\in V(\OmM)$ and test with $\mathbf v_h:=\Pi_h^F\mathbf v\in V_h$:
\[
\int_{\OmM}q_h\,\nabla\cdot\mathbf v\,dx
=\int_{\OmM}q_h\,\nabla\cdot\mathbf v_h\,dx
\le \norm{q_h}_{*,\OmM,h}\,\norm{\nabla\mathbf v_h}_{L^2(\OmM)}
\le C_F\,\norm{q_h}_{*,\OmM,h}\,\norm{\nabla\mathbf v}_{L^2(\OmM)}.
\]
Taking the supremum over $\mathbf v$ gives $\norm{q_h}_{*,\OmM}\le C_F\,\norm{q_h}_{*,\OmM,h}$, and Theorem~\ref{thm:main} then yields $c\,\norm{q_h}_{L^2+\sigma_\eps H^1}\le C_F\,\norm{q_h}_{*,\OmM,h}$.
\end{proof}

Corollary~\ref{cor:discrete_equiv} is the object we feed into the operator-preconditioning framework: from here on every constant is independent of $\eps$, $m$ and $h$.

\subsection{Discretized time-dependent Stokes flow}\label{subsec:time_discrete_stokes}

Theorem~\ref{thm:main} treats the steady Stokes inf-sup norm on $\OmM$. We now consider one backward-Euler step of the time-dependent Stokes equations. The same pressure-extension argument yields a norm equivalence uniform in the step size $\tau>0$. The only additional input is the $L^2$ stability of the restriction operator.

\paragraph{The time-discrete inf-sup norm.}
For a step size $\tau>0$, define on $Q(\OmM)$ the \emph{time-discrete Stokes inf-sup norm}
\begin{equation}
\label{eq:time_discrete_infsup_def}
    \norm{q}_{*,\OmM,\tau}
    \;:=\;
    \sup_{\mathbf v\in V(\OmM)\setminus\{\mathbf 0\}}
    \frac{
        \int_{\OmM} q\,\nabla\!\cdot\!\mathbf v\,dx
    }{
        \bigl(
            \norm{\nabla\mathbf v}_{L^2(\OmM)}^2
            +(\tau\mu)^{-1}\norm{\mathbf v}_{L^2(\OmM)}^2
        \bigr)^{1/2}
    }.
\end{equation}
The denominator is the velocity energy norm induced by a backward-Euler step,
\[
    \tau^{-1}\mathbf u-\mu\Delta\mathbf u+\nabla p=\mathbf f,
\]
after division by $\mu$. As $\tau\to\infty$, it reduces to the steady Stokes inf-sup norm \eqref{eq:infsup_def}.

\paragraph{An $L^2$-stable restriction operator.}
The cellwise construction of \cite{Allaire1991} supplies, in
addition to the gradient estimate \eqref{eq:restriction_operator_estimate},
the $L^2$ stability bound
\begin{equation}
\label{eq:R_L2_stability}
    \norm{R_\eps\mathbf\Phi}_{L^2(\OmM)}
    \;\le\;
    C_R^{L^2}\,\norm{\mathbf\Phi}_{L^2(\Omega)},
    \qquad
    \forall\mathbf\Phi\in H_0^1(\Omega)^d,
\end{equation}
with $C_R^{L^2}$ independent of $\eps$. Property
\eqref{eq:R_L2_stability} is essentially elementary: in each cell
$P_\eps^{i,j}$, $R_\eps\mathbf\Phi$ is built by a local correction whose
$L^2$ norm is controlled by the $L^2$ norm of $\mathbf\Phi$ on the same
cell. We adopt \eqref{eq:R_L2_stability} for the remainder of this
subsection.

\paragraph{The combined short scale.}
Set
\begin{equation}
\label{eq:combined_sigma_tau}
    \sigma_{\eps,\tau}
    \;:=\;
    \bigl(\sigma_\eps^{-2}+(\tau\mu)^{-1}\bigr)^{-1/2}.
\end{equation}
Equivalently $\sigma_{\eps,\tau}\asymp\min(\sigma_\eps,(\tau\mu)^{1/2})$. Thus, for $\tau\gg\sigma_\eps^2/\mu$, the perforation supplies the short scale, whereas for $\tau\ll\sigma_\eps^2/\mu$, the time-discretization term supplies it.

\begin{theorem}[Time-discrete Stokes norm equivalence]\label{thm:time_discrete_stokes}
Under the geometric assumptions of Section~\ref{sec:problem} and the
$L^2$ stability \eqref{eq:R_L2_stability}, there exist constants $c,C>0$
independent of $\eps$ \emph{and} of $\tau>0$ such that for every
$q\in Q(\OmM)$,
\begin{equation}
\label{eq:time_discrete_equivalence}
    c\,\norm{q}_{L^2+\sigma_{\eps,\tau} H^1}
    \;\le\;
    \norm{q}_{*,\OmM,\tau}
    \;\le\;
    C\,\norm{q}_{L^2+\sigma_{\eps,\tau} H^1}.
\end{equation}
\end{theorem}

\begin{proof}[Sketch]
The argument tracks the proof of Theorem~\ref{thm:main} with $\sigma_\eps$
replaced by $\sigma_{\eps,\tau}$ throughout. We highlight only the steps
that change.

\smallskip
\noindent\textbf{Upper bound.}
For $\mathbf v\in V(\OmM)$, the perforated Poincar\'e inequality
$\norm{\mathbf v}_{L^2(\OmM)}\le C_P\sigma_\eps\norm{\nabla\mathbf v}_{L^2(\OmM)}$
\cite{infsup_perforation} yields the mass-energy bound
\begin{equation}
\label{eq:mass_energy_bound}
    \sigma_{\eps,\tau}^{-2}\,\norm{\mathbf v}_{L^2(\OmM)}^2
    \;=\;
    \sigma_\eps^{-2}\,\norm{\mathbf v}_{L^2(\OmM)}^2
    +(\tau\mu)^{-1}\,\norm{\mathbf v}_{L^2(\OmM)}^2
    \;\le\;
    C_P^2\,\norm{\nabla\mathbf v}_{L^2(\OmM)}^2
    +(\tau\mu)^{-1}\,\norm{\mathbf v}_{L^2(\OmM)}^2.
\end{equation}
For any $\tilde q\in H^1(\OmM)$, Cauchy--Schwarz gives
\(
    \int_{\OmM}(q-\tilde q)\nabla\!\cdot\!\mathbf v
    \le \sqrt d\norm{q-\tilde q}_{L^2}\norm{\nabla\mathbf v}_{L^2},
\)
and integration by parts with \eqref{eq:mass_energy_bound} gives
\(
    \int_{\OmM}\tilde q\nabla\!\cdot\!\mathbf v
    = -\int_{\OmM}\nabla\tilde q\cdot\mathbf v
    \le \norm{\sigma_{\eps,\tau}\nabla\tilde q}_{L^2}\cdot\sigma_{\eps,\tau}^{-1}\norm{\mathbf v}_{L^2}
    \le \sqrt{1+C_P^2}\,\norm{\sigma_{\eps,\tau}\nabla\tilde q}_{L^2}
    (\norm{\nabla\mathbf v}_{L^2}^2+(\tau\mu)^{-1}\norm{\mathbf v}_{L^2}^2)^{1/2}.
\)
Taking the supremum over $\mathbf v$, the infimum over $\tilde q$, and combining,
\[
    \norm{q}_{*,\OmM,\tau}
    \le
    \sqrt{d+1+C_P^2}\,
    \norm{q}_{L^2+\sigma_{\eps,\tau}H^1}.
\]

\smallskip
\noindent\textbf{Lower bound.}
Define $\ell_q$ on $H_0^1(\Omega)^d$ by
\eqref{eq:pressure_extension_functional} as in
Section~\ref{subsec:restriction_extension}. The time-discrete inf-sup definition
\eqref{eq:time_discrete_infsup_def} applied to $R_\eps\mathbf\Phi$ gives
\(
    |\ell_q(\mathbf\Phi)|
    \le
    \norm{q}_{*,\OmM,\tau}\,
    (\norm{\nabla R_\eps\mathbf\Phi}_{L^2(\OmM)}^2
        +(\tau\mu)^{-1}\norm{R_\eps\mathbf\Phi}_{L^2(\OmM)}^2)^{1/2}.
\)
By \eqref{eq:restriction_operator_estimate} and \eqref{eq:R_L2_stability},
\[
    \norm{\nabla R_\eps\mathbf\Phi}_{L^2(\OmM)}
    +(\tau\mu)^{-1/2}\norm{R_\eps\mathbf\Phi}_{L^2(\OmM)}
    \le
    C_R\bigl(\norm{\nabla\mathbf\Phi}_{L^2(\Omega)}+\sigma_\eps^{-1}\norm{\mathbf\Phi}_{L^2(\Omega)}\bigr)
    +C_R^{L^2}(\tau\mu)^{-1/2}\norm{\mathbf\Phi}_{L^2(\Omega)},
\]
and using $\sigma_\eps^{-1}+(\tau\mu)^{-1/2}\le\sqrt 2\,\sigma_{\eps,\tau}^{-1}$
we obtain the $\sigma_{\eps,\tau}$-scaled dual bound
\begin{equation}
\label{eq:ell_q_time_discrete_bound}
    |\ell_q(\mathbf\Phi)|
    \;\le\;
    C\,\norm{q}_{*,\OmM,\tau}
    \bigl(
        \norm{\nabla\mathbf\Phi}_{L^2(\Omega)}
        +\sigma_{\eps,\tau}^{-1}\norm{\mathbf\Phi}_{L^2(\Omega)}
    \bigr).
\end{equation}
The compatibility condition of Lemma~\ref{lem:de_rham} on $\ell_q$ is
identical to the steady case (property~(2) of
Proposition~\ref{prop:restriction_operator_lu_allaire} does not depend on
$\tau$), so de Rham furnishes the same extension
$Q=\mathcal E_\eps q\in L^2(\Omega)/\mathbb R$. Applying
Lemma~\ref{lem:weighted_necas_K} with parameter $\sigma_{\eps,\tau}$ and
\eqref{eq:ell_q_time_discrete_bound} gives
$\norm{Q}_{K_{\sigma_{\eps,\tau}}(\Omega)}\le C\,\norm{q}_{*,\OmM,\tau}$.
The consistency $Q|_{\OmM}=q$
(Lemma~\ref{lem:pressure_extension_consistency}) is unchanged, and the
$H^1$-competitor restriction of Section~\ref{sec:lower_bound_proof} delivers
the bound on $\norm{q}_{L^2+\sigma_{\eps,\tau}H^1(\OmM)}$.
\end{proof}

\begin{corollary}[Time-discrete Stokes preconditioner]
\label{cor:time_discrete_stokes_precond}
For a backward-Euler step of size $\tau>0$ applied to the time-dependent
Stokes problem
$\partial_t\mathbf u-\mu\Delta\mathbf u+\nabla p=\mathbf f$,
$\nabla\!\cdot\!\mathbf u=0$ on $\OmM$ with no-slip, the pressure-norm
equivalence \eqref{eq:time_discrete_equivalence} holds with the combined short scale
\(
    \sigma_{\eps,\tau}
    :=\bigl(\sigma_\eps^{-2}+(\tau\mu)^{-1}\bigr)^{-1/2}.
\)
The $K$-functional Riesz map identity of
Section~\ref{subsec:riesz_derivation} adapts to
\begin{equation}
\label{eq:MK_inverse_time_discrete}
    M_K^{-1}
    \;=\;
    M_p^{-1}+\sigma_{\eps,\tau}^{-2}\,L_p^{-1}
    \;=\;
    M_p^{-1}+\bigl(\sigma_\eps^{-2}+(\tau\mu)^{-1}\bigr)\,L_p^{-1},
\end{equation}
so the upper block-triangular preconditioner of
Section~\ref{sec:preconditioner} inherits a $\tau$-uniform iteration count
in addition to the $\eps$- and $h$-uniformity of Theorem~\ref{thm:main}.
The threshold $\tau\asymp\sigma_\eps^2/\mu$ separates the Stokes-like regime
$\tau\gg\sigma_\eps^2/\mu$ (where $\sigma_{\eps,\tau}\sim\sigma_\eps$ and
\eqref{eq:MK_inverse_time_discrete} reduces to \eqref{eq:MK_inverse}) from the
time-dominated regime $\tau\ll\sigma_\eps^2/\mu$ (where
$\sigma_{\eps,\tau}\sim(\tau\mu)^{1/2}$ and the mass term in
\eqref{eq:MK_inverse_time_discrete} dominates).
\end{corollary}

A discrete counterpart of Theorem~\ref{thm:time_discrete_stokes} follows immediately
from Corollary~\ref{cor:discrete_equiv} by repeating the Fortin-operator
argument with the time-discrete Stokes energy denominator; numerical verification of
the $\tau$-uniformity predicted by Corollary~\ref{cor:time_discrete_stokes_precond} is
reported in Section~\ref{subsec:time_discrete_num}.

\section{Norm-induced Preconditioners}\label{sec:preconditioner}

Theorem~\ref{thm:main} and its discrete counterpart Corollary~\ref{cor:discrete_equiv} identify the inf-sup Riesz map on $Q_h$ with the $K$-functional Riesz map up to constants independent of the pillar density $m$ and the mesh size $h$. We use this identification to build an upper block-triangular preconditioner for the discrete Stokes system, following the operator-preconditioning framework of Mardal--Winther~\cite{Mardal2010} and Loghin--Wathen~\cite{loghin2004analysis}; for the broader theory of saddle-point solvers and block preconditioners we refer to the surveys of Benzi, Golub and Liesen~\cite{Benzi2005} and Elman, Silvester and Wathen~\cite{elman2014finite}. The construction inherits the constants of Theorem~\ref{thm:main}, so the resulting GMRES convergence rate is independent of the pillar density $m$ and the mesh size $h$.

\subsection{Discrete saddle-point system}

A conforming, inf-sup stable finite element discretization of the Stokes problem \eqref{eq:stokes_strong} produces the saddle-point system
\begin{equation}\label{eq:discrete_saddle}
    \mathcal{A}\begin{bmatrix}\mathbf{u}\\ p\end{bmatrix}
    \;:=\;
    \begin{bmatrix} A & B^\top\\ B & 0 \end{bmatrix}
    \begin{bmatrix}\mathbf{u}\\ p\end{bmatrix}
    \;=\;
    \begin{bmatrix}\mathbf{f}\\ 0\end{bmatrix},
\end{equation}
where $A\in\R^{n_u\times n_u}$ is the symmetric positive definite Galerkin matrix of the vector Laplacian, $B\in\R^{n_p\times n_u}$ is the discrete divergence operator, and $\mathbf{u}, p$ are the velocity and pressure coefficient vectors. The pressure-space Schur complement is
\begin{equation}\label{eq:schur}
    S \;:=\; B A^{-1} B^\top \;\in\;\R^{n_p\times n_p}.
\end{equation}

\subsection{Schur complement as the inf-sup Riesz map}

The matrix $S$ is the Galerkin representation of the discrete inf-sup norm. For any pressure coefficient vector $p$ with associated finite element function $p_h$,
\begin{equation}\label{eq:S_infsup}
    (Sp,p)
    \;=\; (A^{-1}B^\top p,\, B^\top p)
    \;=\; \sup_{\mathbf{v}\neq\mathbf{0}}\frac{(B^\top p,\,\mathbf{v})^2}{(A\mathbf{v},\,\mathbf{v})}
    \;=\; \norm{p_h}_{*,\OmM,h}^2.
\end{equation}
By Corollary~\ref{cor:discrete_equiv}, $\norm{p_h}_{*,\OmM,h}\asymp\norm{p_h}_{L^2+\sigma_\eps H^1}$ with constants uniform in $\eps$ and $h$. Hence
\begin{equation}\label{eq:S_K_equiv}
    c_*^2\, M_K \;\le\; S \;\le\; C_*^2\, M_K,
\end{equation}
where $M_K\in\mathbb R^{n_p\times n_p}$ is the Galerkin matrix of the $K$-functional norm $\norm{\cdot}_{L^2+\sigma_\eps H^1}$, i.e.\ the unique symmetric positive definite matrix with $\mathbf q^\top M_K\mathbf q=\norm{q_h}_{L^2+\sigma_\eps H^1}^2$ for the finite element function $q_h$ associated with $\mathbf q$. The Schur complement therefore inherits the spectrum of the $K$-functional Riesz map, and any tractable spectral equivalent of $M_K$ becomes a tractable preconditioner for $S$.

\subsection{Closed form of the \texorpdfstring{$K$}{K}-functional Riesz map}\label{subsec:riesz_derivation}

The previous subsection identifies $M_K$ as the spectrally correct pressure preconditioner, but leaves its algebraic form implicit. We now derive that form explicitly. Let $M_p,L_p\in\R^{n_p\times n_p}$ denote, respectively, the pressure mass and Neumann stiffness matrices,
\[
(M_p)_{ij}=\int_{\OmM}\psi_i\psi_j\,dx,
\qquad
(L_p)_{ij}=\int_{\OmM}\nabla\psi_i\cdot\nabla\psi_j\,dx,
\]
where $\{\psi_i\}$ is a basis of $Q_h$; throughout this subsection we also work modulo the one-dimensional constant kernel, so $L_p$ is SPD on $\mathrm{range}(\mathbf q\mapsto \mathbf q - (\mathbf 1^\top M_p\mathbf q/\mathbf 1^\top M_p\mathbf 1)\mathbf 1)$.

By definition,
\begin{equation}\label{eq:M_K_variational}
    \mathbf q^\top M_K \mathbf q
    \;=\;\norm{q_h}_{L^2+\sigma_\eps H^1}^2
    \;=\;\inf_{\tilde{\mathbf q}\in\R^{n_p}}
    \underbrace{(\mathbf q-\tilde{\mathbf q})^\top M_p(\mathbf q-\tilde{\mathbf q})
    + \sigma_\eps^2\,\tilde{\mathbf q}^\top L_p\tilde{\mathbf q}}_{=:F(\tilde{\mathbf q};\mathbf q)}.
\end{equation}
Setting $\nabla_{\tilde{\mathbf q}} F=0$ gives the discrete Euler--Lagrange equation
\begin{equation}\label{eq:EL_discrete}
    (M_p+\sigma_\eps^2 L_p)\,\tilde{\mathbf q}_* \;=\; M_p\,\mathbf q,
    \qquad
    \tilde{\mathbf q}_* = (M_p+\sigma_\eps^2 L_p)^{-1} M_p\,\mathbf q,
\end{equation}
which is the Galerkin discretization of the Neumann reaction--diffusion equation $\tilde q_* - \sigma_\eps^2 \Delta \tilde q_* = q$ with $\partial_\nu \tilde q_*=0$. Substituting \eqref{eq:EL_discrete} back into \eqref{eq:M_K_variational} and using the optimality identity $M_p(\mathbf q-\tilde{\mathbf q}_*)=\sigma_\eps^2 L_p\tilde{\mathbf q}_*$,
\begin{equation}\label{eq:MK_form1}
    \mathbf q^\top M_K\mathbf q
    = (\mathbf q-\tilde{\mathbf q}_*)^\top M_p\mathbf q
    = \mathbf q^\top\bigl(M_p - M_p(M_p+\sigma_\eps^2 L_p)^{-1}M_p\bigr)\mathbf q.
\end{equation}
Writing $N:=M_p+\sigma_\eps^2 L_p$ and using $M_p = N - \sigma_\eps^2 L_p$,
\begin{equation}\label{eq:MK_form2}
    M_K \;=\; M_p - M_p N^{-1} M_p
    \;=\; \sigma_\eps^2 L_p\,N^{-1}\,M_p
    \;=\; M_p\,N^{-1}\,\sigma_\eps^2 L_p,
\end{equation}
where the two symmetric-looking factorizations are equal because $M_K$ is symmetric. The matrix $M_K$ is therefore the \emph{parallel sum} of $M_p$ and $\sigma_\eps^2 L_p$ (the discrete analogue of the harmonic mean of operators). Inverting \eqref{eq:MK_form2} via the identity $(M_p-M_pN^{-1}M_p)^{-1}=M_p^{-1}+(\sigma_\eps^2 L_p)^{-1}$ --- which holds because $M_K\cdot(M_p^{-1}+\sigma_\eps^{-2}L_p^{-1})=I-M_pN^{-1}+M_pN^{-1}=I$ by a direct computation --- we obtain the closed form
\begin{equation}\label{eq:MK_inverse}
    {\,M_K^{-1} \;=\; M_p^{-1} \;+\; \sigma_\eps^{-2}\,L_p^{-1}\,.}
\end{equation}
Equation~\eqref{eq:MK_inverse} is the central structural identity: the Riesz map of the $K$-functional norm, viewed as an operator on $Q_h$, is the sum of the $L^2$ Riesz map $M_p^{-1}$ and the $\sigma_\eps^{-2}$-scaled Neumann-Laplacian inverse $\sigma_\eps^{-2}L_p^{-1}$. The two summands capture the two extreme regimes of the $K$-functional: on low-frequency pressure modes ($\sigma_\eps^2\lambda_k\gg 1$ in the generalized eigenbasis $L_p\mathbf v_k=\lambda_k M_p\mathbf v_k$) the $M_p^{-1}$ term dominates and $M_K\sim M_p$, consistent with the sum norm being $L^2$-like there; on high-frequency modes ($\sigma_\eps^2\lambda_k\ll 1$) the $\sigma_\eps^{-2}L_p^{-1}$ term dominates and $M_K\sim\sigma_\eps^2 L_p$, consistent with the sum norm collapsing onto the $\sigma_\eps H^1$ scale. The identity \eqref{eq:MK_inverse} is also consistent with the continuous $K$-method: on $L^2\cap H^1$ the sum-norm dual is the intersection-norm dual of the gradient of $q$, which at the operator level is the parallel-sum inverse $(\mathrm{id})^{-1}+(-\sigma_\eps^2\Delta_N)^{-1}$ of the Neumann reaction--diffusion operator.

\subsection{Upper block-triangular preconditioner}

The operator-preconditioning principle prescribes that any spectrally equivalent surrogate for $S$ yields a robust preconditioner for $\mathcal{A}$~\cite{Mardal2010,loghin2004analysis}. We adopt the upper block-triangular form
\begin{equation}\label{eq:Pu}
    \mathcal{P}_U \;:=\;
    \begin{bmatrix}
        \hat{A} & B^\top \\
        0       & -\hat{S}
    \end{bmatrix},
\end{equation}
where $\hat{A}$ is a spectrally equivalent approximation of $A$ and $\hat{S}$ is a spectrally equivalent approximation of the Schur complement $S$, which by \eqref{eq:S_K_equiv}--\eqref{eq:MK_inverse} is spectrally equivalent to $M_K$.

The velocity block is standard: $\hat{A}^{-1}$ is realized by a single algebraic multigrid V-cycle~\cite{RugeStuben1987,VanekMandelBrezina1996}, $\hat{A}^{-1}=\mathrm{AMG}(A)$, which is robust because $A$ is a uniformly elliptic vector Laplacian.

The pressure block is realized directly from \eqref{eq:MK_inverse}. Since $M_K^{-1}$ splits as the sum $M_p^{-1}+\sigma_\eps^{-2}L_p^{-1}$, an action $\hat S^{-1}r$ is obtained by adding an $L^2$-inverse and a scaled Laplacian inverse:
\begin{equation}\label{eq:Shat}
    \hat{S}^{-1}\;:=\;\mathrm{AMG}(M_p)\;+\;\sigma_\eps^{-2}\,\mathrm{AMG}(L_p),
\end{equation}
where each of the two summands is realized by one hypre/BoomerAMG V-cycle~\cite{HensonYang2002}. The mass matrix inverse $\mathrm{AMG}(M_p)$ can equivalently be replaced by a Jacobi-preconditioned Chebyshev smoother, which is sometimes cheaper; we use BoomerAMG in both blocks for uniformity. The cost of one application of $\hat S^{-1}$ is therefore two AMG V-cycles on the pressure space, both linear in $n_p$.

\subsection{Spectral bound for the preconditioned operator}

\begin{proposition}\label{prop:precond_bound}
    Let $\hat{A}, \hat{S}$ be symmetric positive definite and satisfy
    \[
        \alpha_A A \le \hat{A} \le \beta_A A,
        \qquad
        \alpha_S S \le \hat{S} \le \beta_S S,
    \]
    with $\alpha_A, \beta_A, \alpha_S, \beta_S > 0$ independent of $\eps$, $m$, and $h$. Then the eigenvalues of $\mathcal{P}_U^{-1}\mathcal{A}$ lie in a bounded set away from the origin whose endpoints depend only on $\alpha_A, \beta_A, \alpha_S, \beta_S$. Preconditioned GMRES applied to \eqref{eq:discrete_saddle} with $\mathcal{P}_U$ converges at a rate determined by these constants alone.
\end{proposition}

Proposition~\ref{prop:precond_bound} is the upper block-triangular case of the Loghin--Wathen analysis~\cite{loghin2004analysis}; we omit the proof. Theorem~\ref{thm:main}, via Corollary~\ref{cor:discrete_equiv} and the identity \eqref{eq:MK_inverse}, supplies the Schur complement constants: $\alpha_S$ and $\beta_S$ depend only on the equivalence constants $c_*, C_*$ between the discrete inf-sup norm and the $K$-functional norm and on the BoomerAMG approximation constants for the uniformly elliptic operators $M_p$ and $L_p$, all of which are uniform in $\eps$ and $m$. The velocity constants $\alpha_A, \beta_A$ are the standard AMG approximation constants for the discrete vector Laplacian. The iteration count of GMRES is therefore independent of $\eps$, the pillar density $m$, and the mesh size $h$. The numerical experiments of Section~\ref{sec:numerical} confirm this prediction.

\section{Numerical Results}\label{sec:numerical}

This section reports numerical experiments on the upper block-triangular preconditioner $\mathcal P_U$ of Section~\ref{sec:preconditioner} with the geometry-scaled pressure block
\[
    \hat S^{-1}=M_p^{-1}+\sigma_\eps^{-2}L_p^{-1},
\]
the discrete realization of the Riesz-map identity \eqref{eq:MK_inverse}. The study is organized around the structural predictions of the theory. A focused two-dimensional campaign (Sections~\ref{subsec:m_robust}--\ref{subsec:failure}) isolates each prediction in turn: the $m$-robustness of the iteration count and the necessity of both summands of \eqref{eq:MK_inverse} (Section~\ref{subsec:m_robust}); the underlying spectral mechanism (Section~\ref{subsec:spectrum}); the $\tau$-uniform time-discrete Stokes extension of Theorem~\ref{thm:time_discrete_stokes} (Section~\ref{subsec:time_discrete_num}); the behaviour on genuine tilted DLD arrays (Section~\ref{subsec:dld}); and the close-packing failure boundary as $\kappa\uparrow\frac12$ (Section~\ref{subsec:failure}).

\subsection{Experimental setup}\label{subsec:setup}

\paragraph{Nondimensionalization.}
The focused two-dimensional experiments are run on the fixed unit cell $\Omega=[0,1]^2$ with pitch $\mathrm{pitch}=1/m$, so that the pore scale is $\eps=\mathrm{pitch}=1/m$ and the linear pillar density is exactly $m$. Each pillar is a disk of relative radius $\kappa$ (radius $\kappa\,\mathrm{pitch}$), giving a throat width $w_{\mathrm{gap}}=(1-2\kappa)\eps$. In accordance with eq.~\eqref{eq:MK_inverse} the short scale passed to the pressure block is the throat form
\[
    \sigma_\eps\propto w_{\mathrm{gap}}=(1-2\kappa)\eps.
\]
The short scale passed to the pressure block is therefore taken proportional to the local throat width; the results below concern its geometric scaling rather than a universal numerical prefactor. The left boundary carries a parabolic inflow, the right boundary is a traction-free outlet, and all exterior walls and pillar surfaces are no-slip.

\paragraph{Discretization and solver.}
The Stokes saddle-point system is discretized with Taylor--Hood $\mathbb P_2$--$\mathbb P_1$ elements on unstructured simplicial meshes, with the throat resolved by at least $3$--$4$ element layers. The system is solved by right-preconditioned FGMRES (relative tolerance $10^{-8}$, restart $1000$), the PETSc~\cite{petsc-user-ref} field split using the upper Schur factorization \eqref{eq:Pu}. The velocity block is preconditioned by one HYPRE BoomerAMG V-cycle (HMIS coarsening, extended+i interpolation, $P_{\max}=4$); the pressure block applies $\hat S^{-1}=M_p^{-1}+\sigma_\eps^{-2}L_p^{-1}$, the mass matrix $M_p$ and the scaled Laplacian $\sigma_\eps^2 L_p$ each solved by one BoomerAMG V-cycle. The per-iteration cost is therefore three AMG V-cycles, all linear in the number of pressure degrees of freedom.

\paragraph{A note on the pressure Laplacian.}
Attaining the $m$-flat iteration counts reported below requires that the pressure Laplacian $L_p$ be assembled with a Dirichlet pin at the outlet (with boundary lifting) rather than as a pure Neumann operator with the constant mode removed by a null-space projection. The two are equivalent in exact arithmetic, but a single BoomerAMG V-cycle on the pinned operator is $m$-uniform, whereas the Neumann/null-space variant degrades with $m$ (e.g.\ from $36$ to $58$ iterations on a representative case). The same pinned $L_p$ is used in the spectral diagnostics of Section~\ref{subsec:spectrum}. The experiments run under DOLFINx~0.9.0~\cite{dolfinx2023} / PETSc~3.23~\cite{petsc-user-ref} (HYPRE~\cite{HensonYang2002}, SLEPc~\cite{HernandezRomanVidal2005}) with Gmsh~4.13~\cite{Gmsh2009}.

\subsection{Pillar density and the necessity of both summands}\label{subsec:m_robust}
The central prediction of Theorem~\ref{thm:main} is that the FGMRES iteration count is bounded independently of the pillar density $m$, in contrast with the $\Theta(m^{-1})$ degradation of the unpreconditioned inf-sup constant $\beta(\Omega_m)$. We verify this and, at the same time, isolate the contribution of each of the two summands of the Riesz map $M_K^{-1}=M_p^{-1}+\sigma_\eps^{-2}L_p^{-1}$ by comparing five pressure-block preconditioners over the range $m=2,\dots,32$ (a $16\times$ span in density) at three pillar radii $\kappa\in\{0.10,0.25,0.40\}$:
\begin{itemize}\setlength\itemsep{0pt}
    \item $M_p^{-1}$ --- the mass inverse alone, corresponding to the Stokes limit;
    \item $\sigma_\eps^{-2}L_p^{-1}$ --- the stiffness inverse alone, corresponding to the Darcy limit;
    \item $M_p^{-1}+\sigma_\eps^{-2}L_p^{-1}$ --- the $K$-functional inverse sum \eqref{eq:MK_inverse};
    \item $(M_p+\sigma_\eps^2 L_p)^{-1}$ --- the operator-sum (series) combination of Meier et al.~\cite{Meier2022}, with the same geometric weight;
    \item $(M_p+c_\star\sigma_\eps^2 L_p)^{-1}$ --- the same operator sum with $c_\star$ selected at $m=8$ and then fixed for all $m$.
\end{itemize}
Figure~\ref{fig:expA} and Table~\ref{tab:expA} report the results. The inverse-sum block is essentially flat: at $\kappa=0.25$ it runs $28\to29$ over $m=2\to32$ ($1.04\times$), and even at the demanding $\kappa=0.40$ it grows only $35\to40$ ($1.14\times$). The ablations fail in complementary ways. The mass inverse alone explodes with $m$ ($42\to{>}1000$ at $\kappa=0.25$), since it omits the high-frequency $\sigma_\eps H^1$ component; the stiffness inverse alone fails to converge at all, since it omits the low-frequency $L^2$ component---the two summands of \eqref{eq:MK_inverse} are individually necessary. The operator sum, although it uses the \emph{same} geometric weight $\sigma_\eps$, degrades with $m$ ($62\to520$ at $\kappa=0.25$), and fixing the best constant weight does not repair it ($41\to214$): the parallel (inverse-sum) structure of the $K$-functional Riesz map, not merely the choice of weight, is what delivers $m$-robustness. This turns the qualitative observation of \cite{Meier2022}---that $L^2$ and $H^1$ pressure norms should be combined---into the quantitative statement that they must be combined \emph{in parallel} at the scale $\sigma_\eps$.

\begin{figure}[t]
    \centering
    \includegraphics[width=\textwidth]{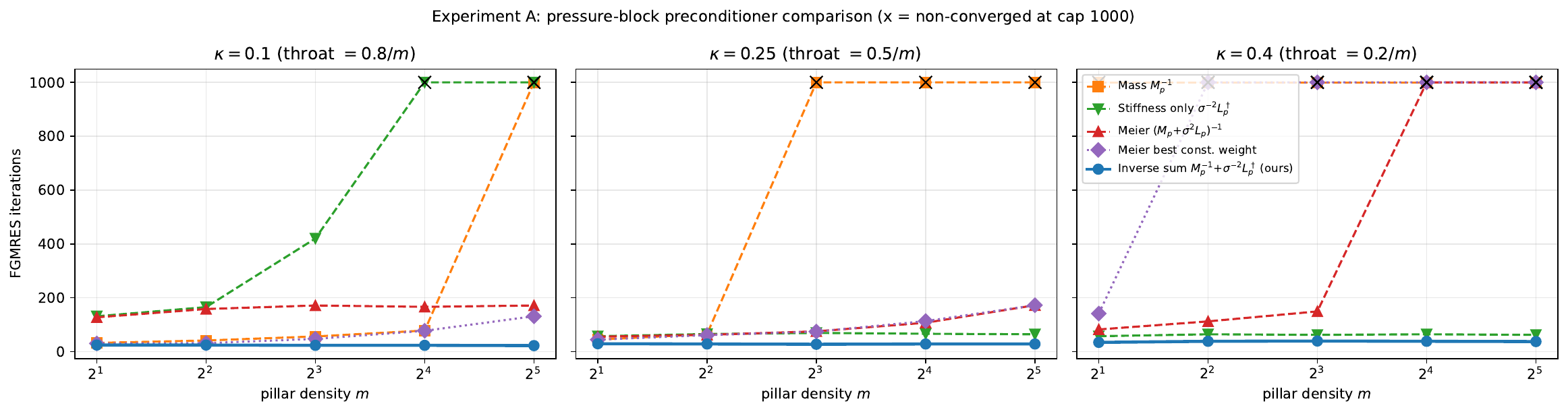}
    \caption{Pillar density robustness and ablation (Section~\ref{subsec:m_robust}). FGMRES iterations versus $m$ for the five pressure-block formulas, one panel per pillar radius $\kappa$; $\times$ marks non-convergence at the iteration cap $1000$. The inverse-sum block $M_p^{-1}+\sigma_\eps^{-2}L_p^{-1}$ (lowest curve) is flat in $m$, while the one-term and operator-sum alternatives degrade.}
    \label{fig:expA}
\end{figure}

\begin{table}[t]
    \centering
    \caption{Ablation of the pressure block at $\kappa=0.25$ (Section~\ref{subsec:m_robust}); ${}^\dagger$ denotes non-convergence at the cap $1000$. Only the parallel inverse sum is flat in $m$.}
    \label{tab:expA}
    \input{latex/exp_A_ablation.tex}
\end{table}

\subsection{Spectral mechanism}\label{subsec:spectrum}
To confirm that the flat iteration counts of Section~\ref{subsec:m_robust} reflect the predicted spectral picture rather than a fortuitous cancellation, we compute, on small serial instances, the dense generalized spectrum on the zero-mean subspace $L_0^2$. Figure~\ref{fig:expB} (left) and Table~\ref{tab:expB} confirm the two ingredients of the unpreconditioned problem: the Schur-complement condition number grows as $\kappa(S)=\Theta(m^2)$ (measured log--log slope $1.86$ at $\kappa=0.25$, $1.71$ at $\kappa=0.40$) and the discrete inf-sup constant decays as $\beta_h=\Theta(m^{-1})$ (slope $-0.89$), exactly the $\Theta(m^{-1})$ degradation of $\beta(\Omega_m)$. Figure~\ref{fig:expB} (right) shows the effect of the preconditioner: the spectrum of $M_K^{-1}S$ is confined to a fixed interval independent of $m$---$[0.12,1.88]$ at $\kappa=0.25$ and $[0.09,2.06]$ at $\kappa=0.40$, with the ratio $\lambda_{\max}/\lambda_{\min}$ varying by under $4\%$ across the whole range. This is the spectral counterpart of the bounded iteration count of Proposition~\ref{prop:precond_bound}: the preconditioner compresses an $m^2$-growing spectral interval into an $m$-independent one.

\begin{figure}[t]
    \centering
    \includegraphics[width=\textwidth]{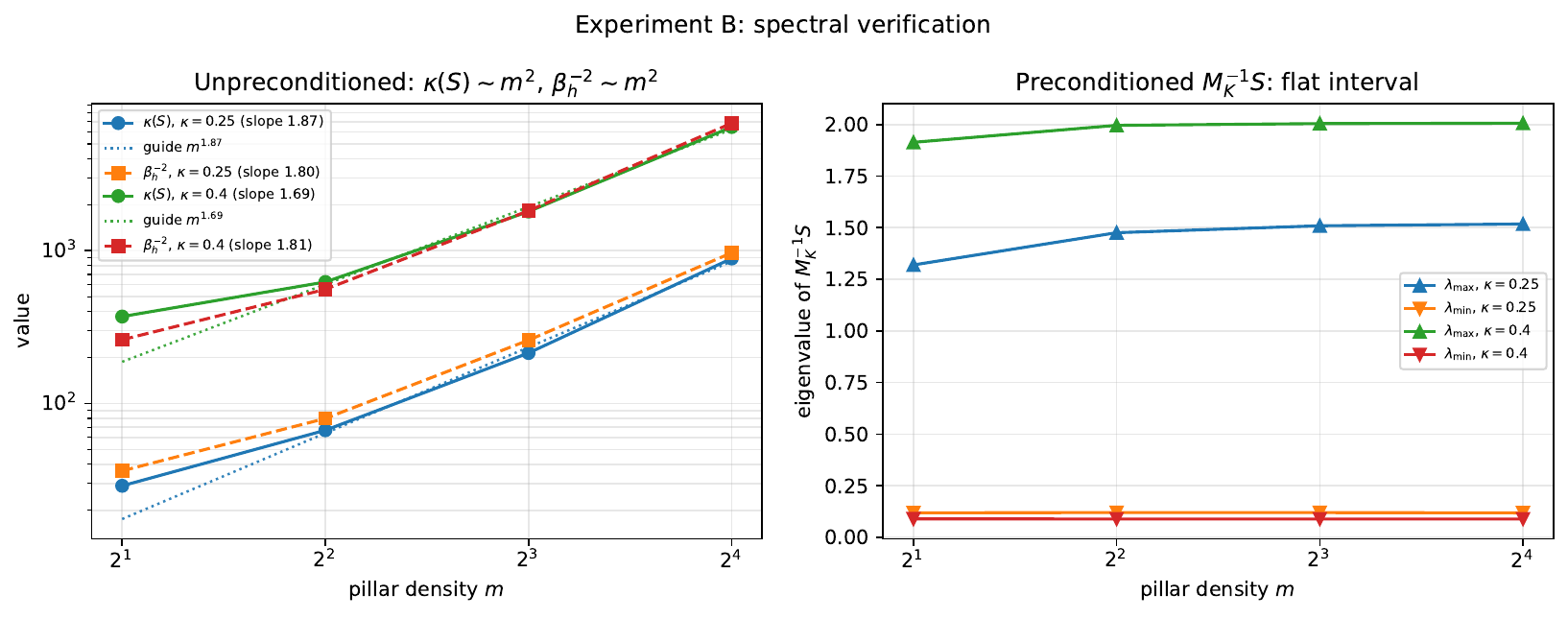}
    \caption{Spectral verification (Section~\ref{subsec:spectrum}). Left: the unpreconditioned Schur condition number $\kappa(S)$ and $\beta_h^{-2}$ grow like $m^2$ (log--log). Right: the preconditioned operator $M_K^{-1}S$ has its spectrum $[\lambda_{\min},\lambda_{\max}]$ confined to a fixed $m$-independent interval.}
    \label{fig:expB}
\end{figure}

\begin{table}[t]
    \centering
    \caption{Spectral data (Section~\ref{subsec:spectrum}): inf-sup constant $\beta_h$, Schur condition number $\kappa(S)$, and the preconditioned interval $[\lambda_{\min},\lambda_{\max}]$ of $M_K^{-1}S$ with its ratio.}
    \label{tab:expB}
    \input{latex/exp_B_spectrum.tex}
\end{table}

\subsection{Discretized time-dependent Stokes: \texorpdfstring{$\tau$}{tau}-uniformity}\label{subsec:time_discrete_num}
Theorem~\ref{thm:time_discrete_stokes} and Corollary~\ref{cor:time_discrete_stokes_precond} extend the equivalence to a backward-Euler discretization of time-dependent Stokes flow. The momentum equation contains the time-discretization term $\tau^{-1}u$, and the pressure block uses the combined short scale $\sigma_{\eps,\tau}=(\sigma_\eps^{-2}+(\tau\mu)^{-1})^{-1/2}$ of \eqref{eq:MK_inverse_time_discrete}. We sweep the step size $\tau$ over seven orders of magnitude ($10^{-4}$ to $10^{3}$) at $\kappa=0.25$, $\mu=1$, and $m\in\{4,8,16,32\}$. The combined scale is uniform across the entire range (fluctuation $1.13$--$1.75\times$), capturing both the small-$\tau$ time-dominated limit and the large-$\tau$ steady Stokes limit.

\paragraph{Both terms are needed across the transition.}
The combined scale is an inverse-sum combination of the geometric term $\sigma_\eps^{-2}$ and the temporal term $(\tau\mu)^{-1}$. At $m=4$, where the threshold $\tau\approx\sigma_\eps^2/\mu$ lies inside the tested range, we compare it with both one-sided controls: $\sigma_\eps$ alone and the time-only scale $(\tau\mu)^{1/2}$. Figure~\ref{fig:expD2}\,(left) shows that the two controls fail at opposite ends of the sweep: $\sigma_\eps$ alone requires $57$ iterations at $\tau=10^{-4}$, while the time-only scale rises to $63$ iterations for large $\tau$. Only the combined scale remains below $29$ iterations throughout.

\paragraph{Small-step density robustness.}
For $\tau\ll\sigma_\eps^2/\mu$, the combined scale satisfies $\sigma_{\eps,\tau}\approx(\tau\mu)^{1/2}$. Figure~\ref{fig:expD2}\,(right) and Table~\ref{tab:expD-density} test this regime over $m\in\{4,8,16,32\}$ at the three smallest time steps. The combined scale remains nearly flat in $m$: for example, it gives $[19,16,20,23]$ iterations at $\tau=10^{-4}$ and a $1.08\times$ span at $\tau=10^{-2}$. The time-only control is accurate at the smallest step but deteriorates once $\tau$ approaches the $m$-dependent threshold, reaching $938$ and $1332$ iterations at $\tau=10^{-2}$ for $m=16$ and $32$.

\begin{figure}[t]
    \centering
    \begin{subfigure}{0.49\textwidth}
        \centering
        \includegraphics[width=\textwidth]{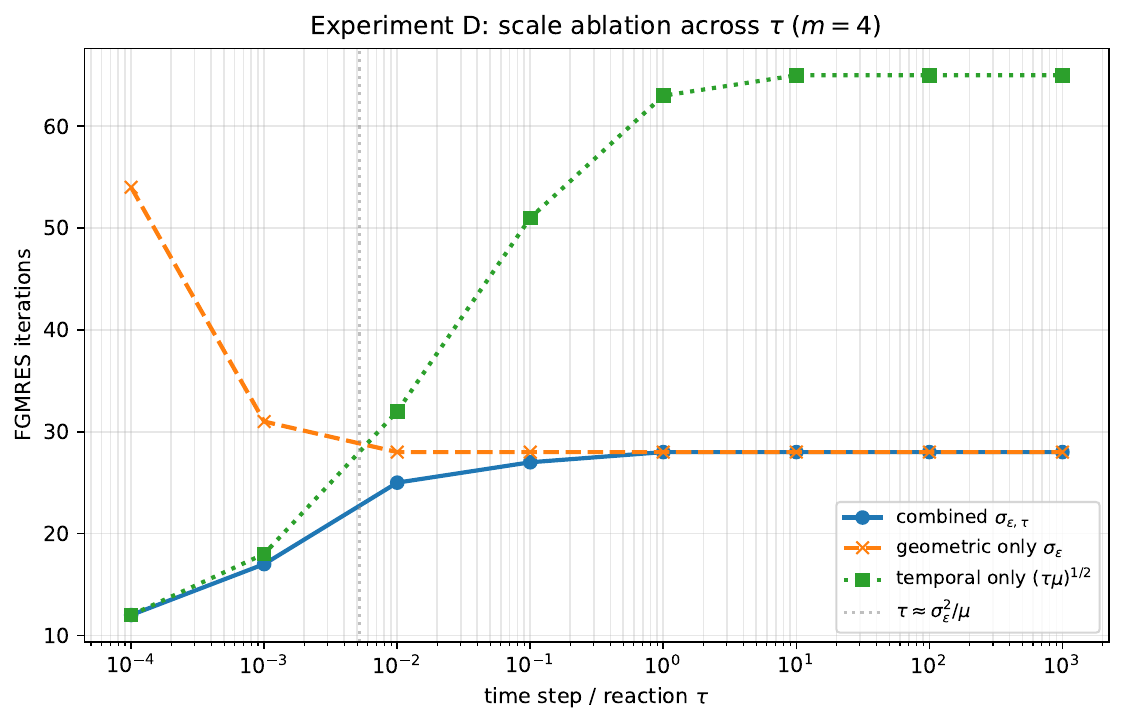}
        \caption{Scale ablation across $\tau$ at $m=4$.}
        \label{fig:expD-ablation}
    \end{subfigure}\hfill
    \begin{subfigure}{0.49\textwidth}
        \centering
        \includegraphics[width=\textwidth]{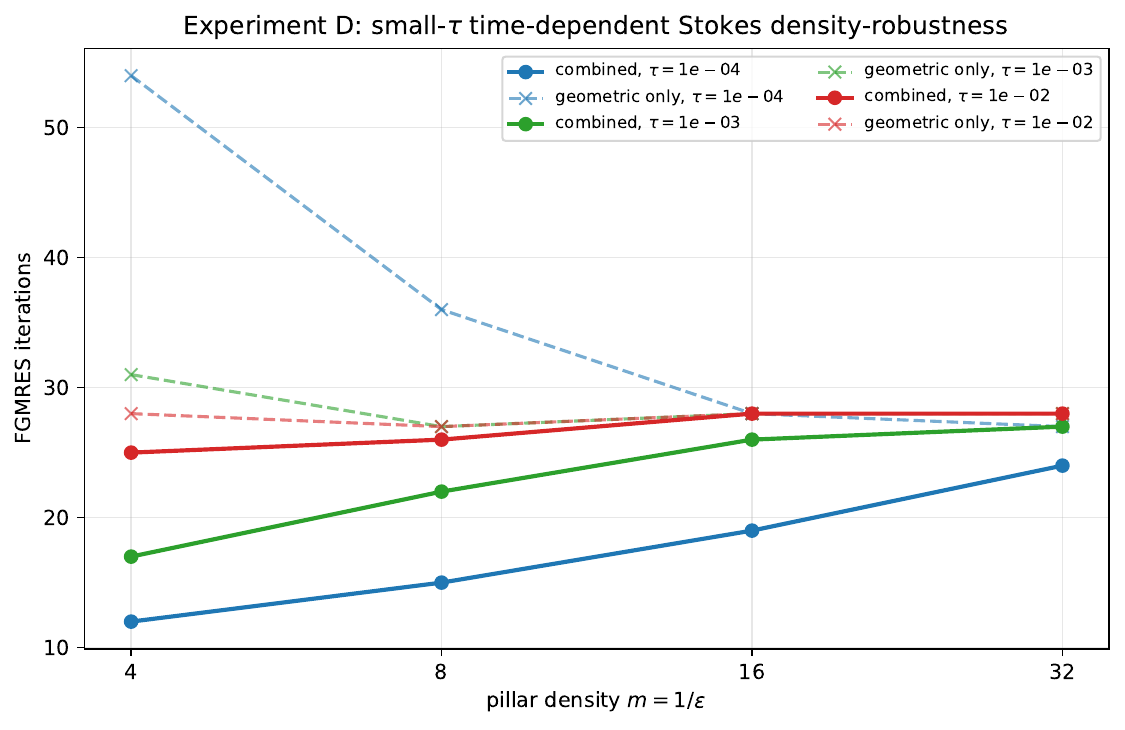}
        \caption{Density robustness at small $\tau$.}
        \label{fig:expD-density}
    \end{subfigure}
    \caption{Time-discrete Stokes validation (Section~\ref{subsec:time_discrete_num}). Left: the combined scale $\sigma_{\eps,\tau}$ remains flat across $\tau$, whereas $\sigma_\eps$ alone fails at small $\tau$ and the time-only scale $(\tau\mu)^{1/2}$ fails at large $\tau$. Right: at small $\tau$, the combined scale is nearly independent of pillar density $m$, while the $\sigma_\eps$-only control degrades.}
    \label{fig:expD2}
\end{figure}

\begin{table}[t]
    \centering
    \caption{Small-$\tau$ density robustness (Section~\ref{subsec:time_discrete_num}): FGMRES iterations versus pillar density $m$, grouped by pressure-block scale.}
    \label{tab:expD-density}
    \begin{subtable}{\textwidth}
        \centering
        \caption{Combined scale $\sigma_{\eps,\tau}$.}
        \input{latex/exp_D_density_combined.tex}
    \end{subtable}
    \medskip
    \begin{subtable}{\textwidth}
        \centering
        \caption{Time-only scale $(\tau\mu)^{1/2}$.}
        \input{latex/exp_D_density_time_only.tex}
    \end{subtable}
    \medskip
    \begin{subtable}{\textwidth}
        \centering
        \caption{Geometric scale $\sigma_\eps$.}
        \input{latex/exp_D_density_geometric.tex}
    \end{subtable}
\end{table}

\subsection{Application: tilted DLD arrays}\label{subsec:dld}
The motivating application throughout this paper is deterministic lateral displacement (DLD), in which the pillar array is tilted by a small row-shift. To confirm that the preconditioner performs on the genuine application geometry---not only on the regular ($\delta=0$) arrays of the preceding experiments---we generate row-shifted arrays in which the centers of row $i$ are displaced by $\delta\,i\,\mathrm{pitch}$, with periodic cropping and inflow/outflow buffer bands keeping the walls clean. We sweep the row-shift fraction $\delta\in\{0,\tfrac1{10},\tfrac15,\tfrac13\}$ at $\kappa\in\{0.25,0.40\}$ and $m\in\{8,16\}$, solving with \texttt{invsum}. Figure~\ref{fig:expE} shows the resulting tilted-lattice flow field, and Table~\ref{tab:expE} reports the iteration counts: the tilted arrays converge in counts comparable to the regular array of the same size ($\kappa=0.25$: $29$ at $\delta=0$ versus $28$--$34$ for $\delta\ne0$; $\kappa=0.40$: $39$--$40$ versus $48$--$61$), all well within the iteration budget. The solver is therefore robust on the target DLD geometry.

\begin{figure}[t]
    \centering
    \begin{subfigure}{0.58\textwidth}
        \centering
        \includegraphics[width=\textwidth]{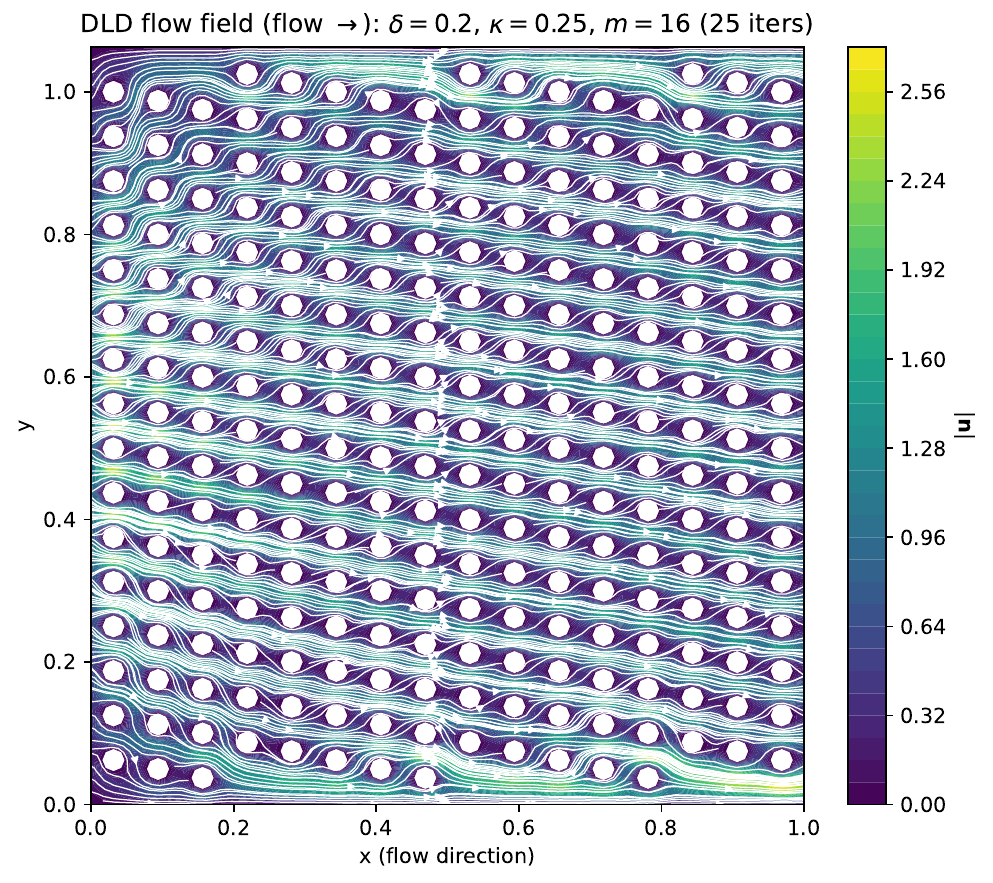}
        \caption{Velocity magnitude and streamlines, $\delta=\tfrac15$, $\kappa=0.25$, $m=16$.}
        \label{fig:expE-field}
    \end{subfigure}\hfill
    \begin{subfigure}{0.40\textwidth}
        \centering
        \includegraphics[width=\textwidth]{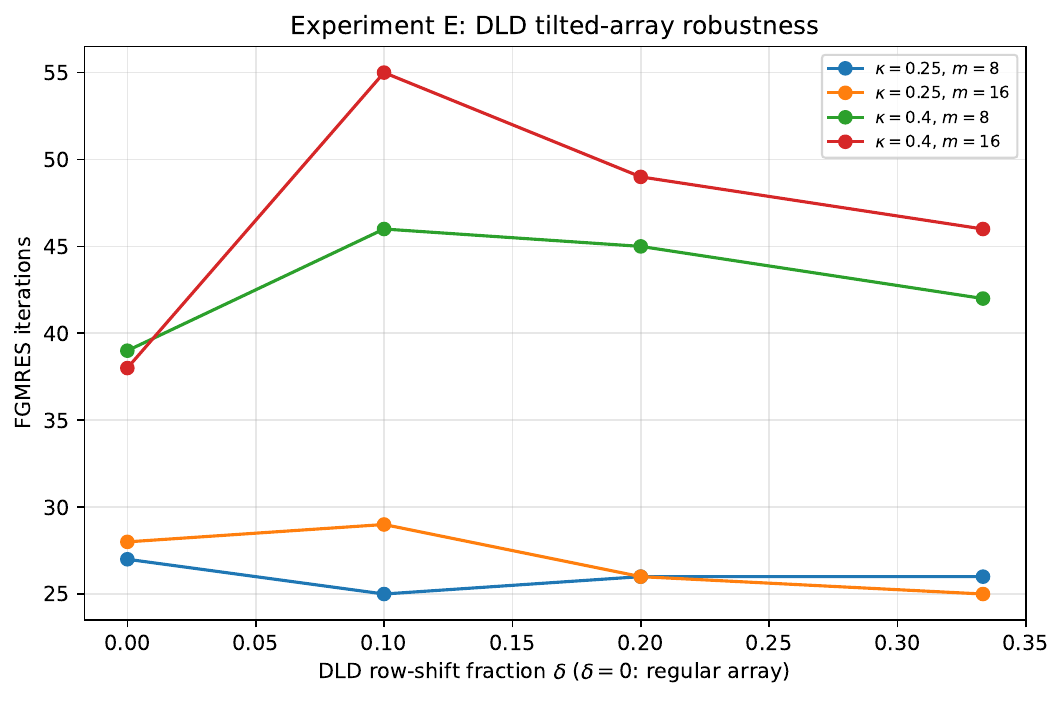}
        \caption{FGMRES iterations versus row-shift $\delta$.}
        \label{fig:expE-iters}
    \end{subfigure}
    \caption{Tilted DLD arrays (Section~\ref{subsec:dld}). The streamlines traverse the row-shifted lattice, and the iteration count on the tilted arrays stays comparable to the regular ($\delta=0$) array of the same size.}
    \label{fig:expE}
\end{figure}

\begin{table}[t]
    \centering
    \caption{FGMRES iterations versus DLD row-shift fraction $\delta$ (Section~\ref{subsec:dld}); $\delta=0$ is the regular array.}
    \label{tab:expE}
    \input{latex/exp_E_dld.tex}
\end{table}

\subsection{The close-packing failure boundary}\label{subsec:failure}
The one geometric quantity the theory does not absorb is the restriction-operator constant $C_R$, which blows up as $\kappa\uparrow\frac12$ (the close-packing limit). We characterize this boundary quantitatively by scanning $\kappa\in\{0.40,0.45,0.47,0.48,0.49\}$ at $m\in\{8,16\}$ with the correct throat scale. Figure~\ref{fig:expF} and Table~\ref{tab:expF} show a super-linear rise in the iteration count toward close-packing (at $m=8$: $35,57,86,121,233$). The count depends only on the relative gap $1-2\kappa$ and not on the pillar density: at each $\kappa$ the $m=8$ and $m=16$ values nearly coincide (e.g.\ $233$ versus $236$ at $\kappa=0.49$), so the two curves of Figure~\ref{fig:expF} overlap. The blow-up follows a power law $N\sim(1-2\kappa)^{-\gamma}$ with exponent $\gamma\approx0.81$, essentially independent of $m$ ($0.82$ at $m=8$ versus $0.81$ at $m=16$). The failure threshold---where the count first exceeds twice the $\kappa=0.40$ baseline---is $\kappa^*\approx0.47$ and does not move with $m$. Crucially, even at $\kappa=0.49$ (throat width $0.00125$ at $m=16$, reaching $3.1\times10^6$ DoFs) the solver still converges ($236$ iterations): the close-packing regime is a quantifiable pre-asymptotic cost, not a hard barrier.

\begin{figure}[t]
    \centering
    \includegraphics[width=\textwidth]{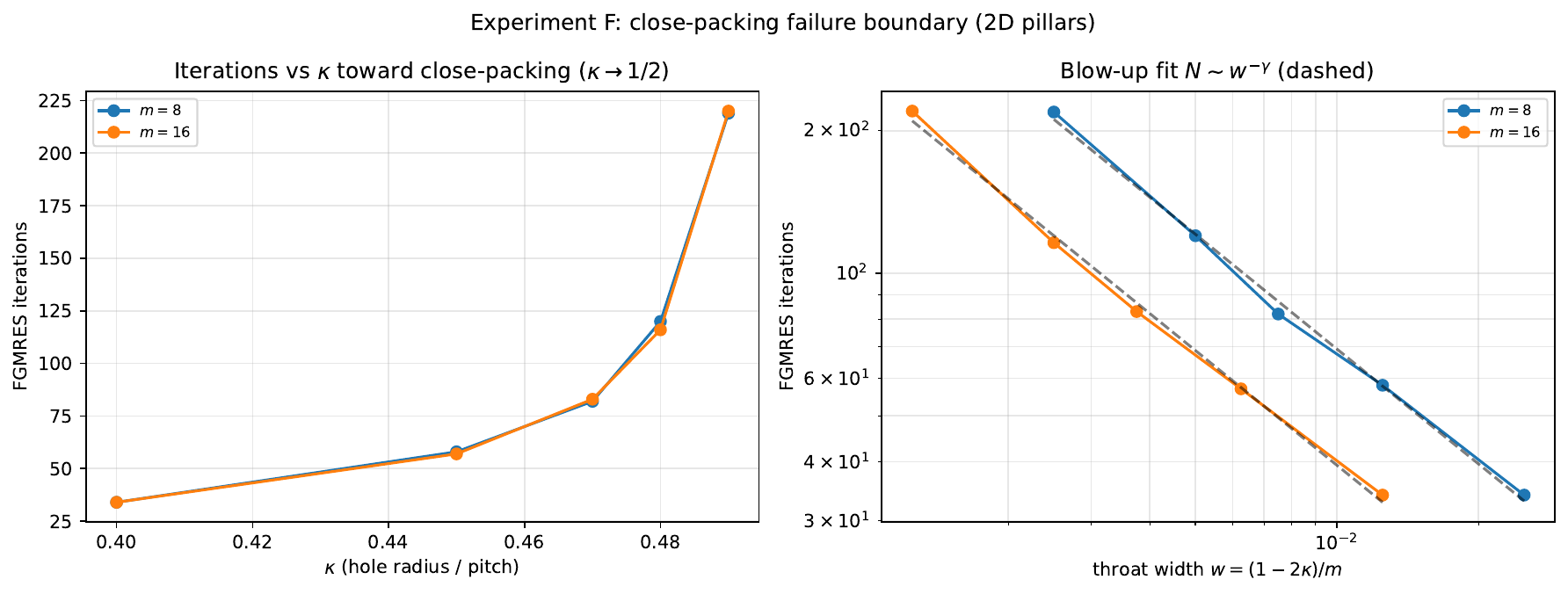}
    \caption{Close-packing failure boundary (Section~\ref{subsec:failure}), for $m=8$ and $m=16$. Left: FGMRES iterations versus $\kappa$ toward $\kappa\to\tfrac12$; the two curves nearly coincide, so the failure threshold is independent of the pillar density $m$. Right: the same data versus the relative gap $1-2\kappa$ (log--log); the two curves collapse onto a single line and the dashed power-law fit $N\sim(1-2\kappa)^{-\gamma}$ has exponent $\gamma\approx0.81$.}
    \label{fig:expF}
\end{figure}

\begin{table}[t]
    \centering
    \caption{FGMRES iterations toward close-packing (Section~\ref{subsec:failure}). The $m=8$ and $m=16$ counts nearly coincide at each $\kappa$, so the blow-up depends only on the relative gap $1-2\kappa$, with fitted exponent $\gamma\approx0.82$ ($m=8$) and $0.81$ ($m=16$). The bottom row is the power law $(1-2\kappa)^{-0.81}$ anchored at $\kappa=0.40$, for reference. All cases converge.}
    \label{tab:expF}
    \input{latex/exp_F_failure.tex}
\end{table}

\subsection{Summary}\label{subsec:numerical_summary}

The two-dimensional experiments confirm the structural predictions of the theory one by one: the iteration count is flat in the pillar density $m$ over a $16\times$ range, and both summands of the Riesz map \eqref{eq:MK_inverse} are individually necessary while their parallel combination---unlike the operator-sum of \cite{Meier2022} at the same weight---delivers $m$-robustness (Section~\ref{subsec:m_robust}); the spectral diagnostics trace this to the compression of an $\Theta(m^2)$ Schur interval into an $m$-independent one (Section~\ref{subsec:spectrum}); the combined scale $\sigma_{\eps,\tau}$ is $\tau$-uniform across the time-discrete transition (Section~\ref{subsec:time_discrete_num}); and the solver is robust on genuine tilted DLD arrays (Section~\ref{subsec:dld}). The only substantial residual dependence is on the pillar radius near close-packing, where the iteration count blows up as $N\sim(1-2\kappa)^{-0.81}$ past $\kappa^*\approx0.47$ but still converges (Section~\ref{subsec:failure}), tracing this to the restriction-operator constant $C_R$. Closing the residual close-packing gap via spatially varying weights that track the local pore aperture is a natural direction for future work.

\section{Conclusion}\label{sec:conclusion}

We have proved that the inf-sup norm $\|\cdot\|_{*,\Omega_m}$ and the $K$-functional norm $\|\cdot\|_{L^2+\sigma_\eps H^1}$ are equivalent on $L_0^2(\Omega_m)$ with constants independent of the periodicity parameter $\epsilon$, on periodic pillar arrays in the proportional-hole regime $a_\eps\sim\eps$. The upper bound follows from the Poincar\'{e} inequality on the perforated domain. The lower bound is established at the PDE level via a dual pressure extension: the restriction operator lifts pressures from $\Omega_m$ to $\Omega$ by duality, a fixed-domain weighted Ne\v{c}as inequality at the $K$-functional scale converts the duality bound into a norm estimate on $\Omega$, and restriction yields the competitor on $\Omega_m$. The argument is mesh-free and bypasses the coarse-partition/quasi-interpolation machinery of \cite{sande2025robust}, which is intrinsic to thin channels and does not transport to a pillar network. Corollary~\ref{cor:discrete_equiv} transports the equivalence to any inf-sup stable finite-element pair admitting an $\eps$-uniform Fortin operator, with constants independent of $\eps$, $m$, and the mesh size $h$. The result lifts the one-sided pressure estimate of \cite{Allaire1991,Lu2020} from a single Stokes pressure to a two-sided norm equivalence valid for arbitrary $q\in L_0^2(\Omega_m)$, and is consistent with the sharp $\Theta(m^{-1})$ degradation of the LBB constant established in \cite{infsup_perforation}.

The equivalence supplies the missing pressure norm for the Mardal--Winther operator-preconditioning framework on perforated geometries. The $K$-functional Riesz map admits the closed form $M_K^{-1}=M_p^{-1}+\sigma_\eps^{-2}L_p^{-1}$ (Section~\ref{subsec:riesz_derivation}), which identifies the perforated Stokes problem with the Brinkman problem at homogenized permeability $\sigma_\eps^2\asymp\eps^2$ and yields a Brinkman-type pressure block --- one $L^2$ AMG solve on $M_p$ plus one scaled Laplacian AMG solve on $L_p$ --- as a robust surrogate for the perforated Stokes Schur complement. Assembled into an upper block-triangular preconditioner $\mathcal P_U$, this gives an algorithm whose per-iteration cost is three AMG V-cycles and is linear in DoFs.

The numerical study confirms the robustness properties predicted by the norm equivalence, prediction by prediction. On focused two-dimensional unit cells, FGMRES iteration counts are flat in the pillar density $m$ over a $16\times$ range, with no systematic growth tracking the $\Theta(m^{-1})$ degradation of $\beta(\Omega_m)$; an ablation shows that both summands of $M_K^{-1}$ are individually necessary and that their parallel (inverse-sum) combination, unlike the operator-sum of \cite{Meier2022} at the same geometric weight, is what yields $m$-robustness. Spectral diagnostics trace this to the compression of an $\Theta(m^2)$ Schur-complement interval into an $m$-independent one. The only substantial residual dependence is on the hole radius near close-packing ($\kappa\gtrsim0.4$), where the iteration count blows up as $N\sim(1-2\kappa)^{-0.81}$ in the relative gap $1-2\kappa$ past a threshold $\kappa^*\approx0.47$ but still converges---a quantifiable pre-asymptotic cost rather than a barrier---and is traced to the geometric blow-up of the restriction-operator constant $C_R$. Closing the residual close-packing gap via spatially varying weights tracking the local pore aperture, or an adaptive penalty in the spirit of \cite{infsup_perforation}, is a natural direction for future work. Finally, the same framework extends to backward-Euler discretizations of time-dependent Stokes flow at the combined scale $\sigma_{\eps,\tau}=(\sigma_\eps^{-2}+(\tau\mu)^{-1})^{-1/2}$ (Theorem~\ref{thm:time_discrete_stokes}), yielding a $\tau$-uniform preconditioner across the time-discrete transition that we confirm numerically in \S\ref{subsec:time_discrete_num}.

\section*{CRediT authorship contribution statement}
\textbf{Qi Xin:}  Methodology, Software, Formal analysis, Investigation, Writing -- original draft.
\textbf{Yan Xie:} Software, Validation.
\textbf{Chensong Zhang:} Methodology,  Funding acquisition.
\textbf{Shihua Gong:} Conceptualization, Formal analysis, Supervision, Writing -- review \& editing, Funding acquisition.
\textbf{Jinchao Xu:} Conceptualization.

\section*{Declaration of competing interest}
The authors declare that they have no known competing financial interests.

\section*{Acknowledgements}
The authors are grateful to Yong Lu (Department of Mathematics, Nanjing University) for valuable advice on homogenization theory.

\section*{Data availability}
No data was used for the research described in the article.

\bibliographystyle{elsarticle-num}
\bibliography{reference}

\end{document}

%% file: latex/exp_A_ablation.tex
\begin{tabular}{lrrrrr}
\hline
$\hat S^{-1}$ & $m=2$ & $m=4$ & $m=8$ & $m=16$ & $m=32$ \\
\hline
$M_p^{-1}$ & 42 & 65 & 699 & 1000$^\dagger$ & 1000$^\dagger$ \\
$\sigma_\epsilon^{-2}L_p^{-1}$ & 1000$^\dagger$ & 1000$^\dagger$ & 1000$^\dagger$ & 1000$^\dagger$ & 1000$^\dagger$ \\
$(M_p+\sigma_\epsilon^2 L_p)^{-1}$ & 62 & 73 & 96 & 141 & 520 \\
$(M_p+c_\star\sigma_\epsilon^2 L_p)^{-1}$ & 41 & 57 & 86 & 141 & 214 \\
$M_p^{-1}+\sigma_\epsilon^{-2}L_p^{-1}$ & 28 & 30 & 29 & 29 & 29 \\
\hline
\end{tabular}

%% file: latex/exp_B_spectrum.tex
\begin{tabular}{llrrrrr}
\hline
$\kappa$ & $m$ & $p$-dofs & $\beta_h$ & $\kappa(S)$ & $[\lambda_{\min},\lambda_{\max}]$ & $\kappa(M_K^{-1}S)$ \\
\hline
0.25 & 2 & 92 & 1.758e-01 & 4.546e+01 & [0.12,1.82] & 15.2 \\
0.25 & 4 & 390 & 1.222e-01 & 9.336e+01 & [0.12,1.82] & 14.9 \\
0.25 & 8 & 1630 & 6.872e-02 & 3.177e+02 & [0.12,1.87] & 15.6 \\
0.25 & 16 & 6680 & 3.573e-02 & 1.233e+03 & [0.12,1.88] & 15.6 \\
\hline
0.4 & 2 & 195 & 6.538e-02 & 3.588e+02 & [0.09,2.01] & 22.4 \\
0.4 & 4 & 672 & 4.611e-02 & 5.474e+02 & [0.09,2.04] & 23.0 \\
0.4 & 8 & 2561 & 2.592e-02 & 1.550e+03 & [0.09,2.05] & 23.2 \\
0.4 & 16 & 10000 & 1.346e-02 & 5.875e+03 & [0.09,2.06] & 23.2 \\
\hline
\end{tabular}

%% file: latex/exp_D_density_combined.tex
\begin{tabular}{lrrrr}
\hline
$\tau$ & $m=4$ & $m=8$ & $m=16$ & $m=32$ \\
\hline
$1e-04$ & 19 & 16 & 20 & 23 \\
$1e-03$ & 18 & 23 & 26 & 26 \\
$1e-02$ & 26 & 27 & 28 & 26 \\
\hline
\end{tabular}

%% file: latex/exp_D_density_time_only.tex
\begin{tabular}{lrrrr}
\hline
$\tau$ & $m=4$ & $m=8$ & $m=16$ & $m=32$ \\
\hline
$1e-04$ & 15 & 14 & 20 & 28 \\
$1e-03$ & 16 & 25 & 35 & 66 \\
$1e-02$ & 31 & 48 & 938 & 1332 \\
\hline
\end{tabular}

%% file: latex/exp_D_density_geometric.tex
\begin{tabular}{lrrrr}
\hline
$\tau$ & $m=4$ & $m=8$ & $m=16$ & $m=32$ \\
\hline
$1e-04$ & 57 & 35 & 29 & 26 \\
$1e-03$ & 30 & 27 & 27 & 26 \\
$1e-02$ & 29 & 28 & 28 & 26 \\
\hline
\end{tabular}

%% file: latex/exp_E_dld.tex
\begin{tabular}{llrrrr}
\hline
$\kappa$ & $m$ & $\delta=0$ & $\delta=0.1$ & $\delta=0.2$ & $\delta=0.333$ \\
\hline
0.25 & 8 & 29 & 29 & 28 & 28 \\
0.25 & 16 & 29 & 34 & 31 & 30 \\
0.4 & 8 & 40 & 50 & 50 & 48 \\
0.4 & 16 & 39 & 61 & 56 & 55 \\
\hline
\end{tabular}

%% file: latex/exp_F_failure.tex
\begin{tabular}{cccccc}
\hline
$\kappa$ & 0.4 & 0.45 & 0.47 & 0.48 & 0.49 \\
\hline
$m=8$ & 35 & 57 & 86 & 121 & 233 \\
$m=16$ & 36 & 59 & 86 & 122 & 236 \\
\hline
$(1-2\kappa)^{-0.81}$ & 35 & 61 & 93 & 129 & 226 \\
\hline
\end{tabular}